%% file: MeanField_final_arxiv_revision.tex
\documentclass[a4wide,10pt]{article}
\usepackage{amsmath}
\usepackage{amssymb}
\usepackage{booktabs}
\usepackage{multirow}
\usepackage{graphicx}
\usepackage{color}
\newcommand{\Id}{\ensuremath{\operatorname{Id}}\,}

\newcommand{\p}{\partial}

 \newcommand{\Z}{{\mathbb Z}}

\newcommand{\M}{{\mathbb M}}

\newcommand{\TT}{{\mathbb T}}

\newcommand{\RX}{\ensuremath{\left]-\infty,+\infty\right]}}
\newcommand{\pinf}{\ensuremath{{+\infty}}}
\newcommand{\dom}{\ensuremath{\operatorname{dom}}}

\newcommand{\minimize}[2]{\ensuremath{\underset{\substack{{#1}}}%
{\mathrm{minimize}}\;\;#2 }}

\newcommand{\scal}[2]{{{({#1})\cdot({#2})}}}
 
\newcommand{\Menge}[2]{\Big\{{#1}~ ;~{#2}\Big\}} 
\newcommand{\menge}[2]{\left\{{#1}~ ; ~{#2}\right\}} 

\newcommand{\prox}{\ensuremath{\operatorname{prox}}}

\newcommand{\T}{{\mathbb{T}}}

\newcommand{\Argmax}[2]{\ensuremath{\underset{\substack{{#1}}}%
{\mathrm{Argmax}}\;\;#2 }}

\newenvironment{proof}[1][Proof]{\textbf{#1.} }{\ \rule{0.5em}{0.5em}}
\newtheorem{remark}{\textbf{Remark}}[section]

\newtheorem{lemma}{\textbf{Lemma}}[section]
\newtheorem{theorem}{\textbf{Theorem}}[section]
\newtheorem{corollary}{\textbf{Corollary}}[section]
\newtheorem{proposition}{\textbf{Proposition}}[section]

\numberwithin{equation}{section}
\title{Proximal methods for stationary Mean Field Games  with local couplings\author{
L. M. Brice\~no-Arias \thanks{Universidad T\'ecnica Federico Santa
Mar\'ia, Departamento de Matem\'atica, Av. Vicu\~na Mackenna 3939,
San Joaqu\'in, Santiago, Chile (luis.briceno@usm.cl).} \and D. Kalise \thanks{Johann Radon Institute for Computational and Applied Mathematics, Austrian Academy of Sciences, Altenbergerstra\ss e 69, A-4040 Linz, Austria (dante.kalise@oeaw.ac.at).} \and F. J.
Silva \thanks{Institut de recherche XLIM-DMI, UMR-CNRS 7252 Facult\'e
des sciences et techniques 
Universit\'e de Limoges, 87060 Limoges, France
(francisco.silva@unilim.fr).}}}
\input macro
\begin{document}
\maketitle
\begin{abstract}
We address  the numerical approximation of Mean Field Games  with 
local couplings. For power-like Hamiltonians, we consider   both the 
stationary system introduced in \cite{LasryLions06i,LasryLions07} and 
also a similar system involving density constraints in order to model 
hard congestion effects \cite{filippo,Mészáros2015}. For finite 
difference discretizations of the Mean Field Game system as in 
\cite{AchdouCapuzzo10}, we follow a variational approach. We prove 
that the aforementioned schemes can be obtained as the optimality 
system of suitably defined optimization problems. In order to prove 
the existence of solutions of the scheme with a variational argument, 
the monotonicity of the coupling term is not used, which allow us to 
recover general existence results proved in \cite{AchdouCapuzzo10}. 
Next, assuming next that the coupling term is monotone, the 
variational problem is cast as a convex optimization problem for 
which we study and compare  several proximal type methods. These 
algorithms have several interesting features, such as global 
convergence and stability with respect to the viscosity parameter, 
which can eventually be zero. We assess the performance of the 
methods via numerical experiments. 
\end{abstract}
\section{Introduction} 
Mean Field Games (MFG) have been recently introduced by J.-M. Lasry 
and P.-L. Lions \cite{LasryLions06i,LasryLions06ii,LasryLions07} and 
by Huang, Caines, and Malham\'e \cite{MR2352434} in order to model 
the behavior of  some differential games when  the number of players 
tends to infinity. For finite-horizon games, and under suitable 
assumptions such as the absence of a common noise affecting 
simultaneously all agents, the description of the limiting 
behaviour collapses into  two coupled deterministic partial 
differential equations (PDEs). The first one is a 
Hamilton-Jacobi-Bellman (HJB) equation with a terminal condition, 
characterizing the value function $v$ of an optimal control problem 
solved by {\it typical small player} and whose cost function depends 
on the distribution $m$ of the other players at each time. The second 
one is a Fokker-Planck (FP) equation, describing, at the Nash 
equilibrium,  the evolution of the initial distribution $m_0$ of the  
agents. In the ergodic case,  the resulting system is   stationary   
and  its solution is the limit of  a rescaled solution of a 
finite-horizon MFG system  when the horizon tends to infinity (see 
\cite{Cardaliagueteetal-2013a,Cardaliagueteetal-2013,Cardaliaguet-2013}
 and also \cite{gomes10} for similar results in the context of 
discrete MFG).  

In order to introduce the system we study, let $\mathbb{T}^n$ be the $n$-dimensional torus and $f: \mathbb{T}^n \times [0,+\infty[ \to \RR$ be a continuous function. Given $\nu\geq 0$ and a   function $H: \mathbb{T}^n\times \RR^{n} \mapsto \RR$, such that for all $x\in  \mathbb{T}^n$ the function $ H(x,\cdot): p\mapsto H(x,p)$ is convex and differentiable, we  consider the following stationary MFG problem: 
find two functions $u$, $m$ and $\lambda \in \RR$ 
such that 
\be\label{mfgestacionario}\ba{rcl} -\nu \Delta u + H(x,\nabla u)  - \lambda &=& f(x,m(x)) \hspace{0.2cm} \mbox{in $\TT^{n}$}, \\[4pt]
- \nu \Delta m -\mbox{div}\left(\partial_{p} H(x,\nabla u)m \right)&=&0  \hspace{0.2cm} 
\mbox{in $\TT^{n}$},\\[5pt]
m\geq 0, \; \; \; \int_{\TT^{n}} m(x) \dd x =1, 
& \; & \; \int_{\TT^{n}} u(x) \dd x =0.
\ea
\ee
When $\nu>0$, well-posedness of system \eqref{mfgestacionario} has been studied in several articles, starting with the works by  J.-M. Lasry and P.-L. Lions \cite{LasryLions06i,LasryLions07}, followed by \cite{MR3160525,MR2928381,MR3333058,MR3415027,gomestal15,pimentel2015regularity} in the case of smooth solutions and \cite{MR3333058,BarFel15,ferreiragomes15,doi:10.1080/03605302.2016.1192647} in the case of weak solutions. 

Let us point out that in terms of the underlying game, system 
\eqref{mfgestacionario} involves {\it local couplings} because the 
right hand side  of \eqref{mfgestacionario} depends on the 
distribution  $m$ through its pointwise value  (see 
\cite{LasryLions07}).  As explained in \cite{LasryLions07}, in this 
case  system \eqref{mfgestacionario} is related to a single optimal 
control problem. Indeed, defining   $b:\RR\times \RR^{n} \mapsto \RR 
\cup \{+\infty\}$ and $F: \TT^{n}\times \RR \mapsto  \RR \cup 
\{+\infty\}$ as \small
\begin{equation}
\label{e:Lq}
b(x,m,w):= \left\{
\begin{array}{ll}
mH^{\ast}\left(x,-\frac{w}{m}\right) & \mbox{if } m>0,\\
0, & \mbox{if } (m,w)=(0,0),\\
+\infty, & \mbox{otherwise}
\end{array}
\right. \hspace{0.1cm}  \; \; F(x,m):= 
\left\{\begin{array}{ll}  \int_{0}^{m} f(x,m') \dd m', & \mbox{if } m \geq 
0, \\[4pt]
+\infty, & \mbox{otherwise,}\end{array}\right.
\end{equation} \normalsize
system  \eqref{mfgestacionario} can be obtained, at least formally, as the optimality system associated to any solution $(m,w)$ of
$$\ba{l}\inf_{(m,w)} \int_{\TT^{n}} \left[b(m(x), w(x)) + F(x,m(x)) 
\right] \dd x,\\[8pt]
\mbox{subject to } \hspace{0.3cm} \ba{rcl} -\nu \Delta m + \mbox{div}(w) &=& 0  
\hspace{0.2cm} \mbox{in $\TT^{n}$},\\[4pt]
								\int_{\TT^{n}} m(x) \dd x&=&1. \ea
\ea \eqno(P)$$
The function $u$  in \eqref{mfgestacionario} corresponds to a Lagrange multiplier associated to the PDE constraint, $\lambda$ is a Lagrange multiplier associated to the integral constraint, and $w$ is given by  $-\partial_{p} H(x,\nabla u)m$.
 Note that the definitions of   $b$ and $F$  involve, implicitly, the non-negativity  of the variable $m$.

In the presence of  {\it hard congestion} effects for the 
agents, we consider  upper bound constraints for the 
density $m$ (see \cite{MauRouSan1,MauRouSan2} for the analysis in the 
context of crowd motion and \cite{filippo} for a proposal in the context of MFG), which we include in 
the following optimization problem  (see \cite{Mészáros2015} for a detailed study)
$$\ba{l}\inf_{(m,w)} \int_{\TT^{n}} \left[b(m(x), w(x)) + F(x,m(x)) 
\right] \dd x,\\[8pt]
\mbox{subject to } \hspace{0.3cm} \ba{rcl} -\nu \Delta m + \mbox{div}(w) &=& 0  
\hspace{0.2cm} \mbox{in $\TT^{n}$},\\[4pt]
								\int_{\TT^{n}} m(x) \dd x&=&1, \; \; \; \; m(x)\leq d (x) \; \; \mbox{a.e. in $\TT^{n}$},\ea
\ea \eqno(P^{d})$$ 
where  $d\in C(\TT^{n})$  satisfies $d(x)>0$ for all $x\in \TT^{n}$ and $\int_{\TT^{n}} d(x) \dd x > 1$.  It is assumed that for all $x\in \TT^n$
\be\label{hstarpowerq} H^{\ast}(x, v)= \frac{1}{q}|v|^{q} \hspace{0.3cm} \mbox{for some $q>1$  and so  } H(x, p)= \frac{1}{q'}|p|^{q'}, \; \; \mbox{where } \; \frac{1}{q}+\frac{1}{q'}=1.\ee
The analysis in \cite{Mészáros2015} is done in the case of a bounded domain $\Om$, with Neumann boundary conditions for the PDE constraint and  $d\equiv 1$.  However, it is easy to check that the results in \cite{Mészáros2015}  can be adapted to our case with minor modifications. If $q>n$, it is shown that there exists at least one solution $(m,w)\in W^{1,q}(\TT^{n})\times L^{q}(\TT^{n})$ to  $(P^{d})$ and there exists $(u,\lambda,p) \in  W^{1,s}(\TT^{n}) \times \RR  \times \M_{+}(\TT^{n})$, where $s \in ]1, n/(n-1)[$ and $\M_+(\TT^{n})$ denotes the set of non-negative Radon measures on $\TT^{n}$,  satisfying
\be\label{mfgestacionarioconrestcaja}\ba{c} -\nu \Delta u + \frac{1}{q'}| 
\nabla u|^{q'} - p-\lambda  \leq  f(x,m(x)) \hspace{0.2cm} \mbox{in $\TT^{n}$}, 
\\[4pt]
		 - \nu \Delta m - \mbox{div}\left( |\nabla 
		 u|^{\frac{2-q}{q-1}}\nabla u\,m\right)=0  \hspace{0.2cm} 
		 \mbox{in $\TT^{n}$},\\[4pt]
						    m\geq 0, \; \; \; \int_{\TT^{n}} m(x) \dd x =1, 
						   \; \; \;  \int_{\TT^{n}} u(x) \dd x =0,\\[6pt]
						     \mbox{supp}(p)\subseteq \{m=1\},
						     \ea
\ee
with the convention
$$ |\nabla u|^{\frac{2-q}{q-1}}\nabla u=0 \hspace{0.4cm} 
\mbox{if } \; q>2 \hspace{0.4cm} \mbox{and } \nabla u=0.$$ 
The first inequality in \eqref{mfgestacionarioconrestcaja} becomes an equality on the set $\{x \in \TT^n \; ; \; m(x)>0\}$. When $1<q\leq n$, an approximation argument shows the existence of solutions of a weak form of \eqref{mfgestacionarioconrestcaja}.

The aim of this work is to consider the numerical approximation 
of solutions of 
\eqref{mfgestacionario}-\eqref{mfgestacionarioconrestcaja} by 
means of their variational formulations.  For the sake of 
simplicity, we restrict our analysis to the 2-dimensional case, 
i.e., we take $n=2$ and consider Hamiltonians of the form 
\eqref{hstarpowerq}. However,    all the results in this work 
admit natural extensions  in  general dimensions and most of 
them are valid for more general Hamiltonians.  We do not 
consider here the case of {\it non-local couplings}, not 
necessarily variational,  and we refer the reader to 
\cite{AchdouCapuzzo10,AchdouCapuzzoCamilli12,camsilfirst11,CS12,CS15,FACK16}
 for some numerical methods for this case. Inspired by 
\cite{AchdouCapuzzoCamilli12}, in the context of MFG systems 
related  to planning problems, we follow the {\it first 
discretize and then optimize strategy} by considering suitable 
finite-difference discretizations of the PDE constraint and the 
cost functionals appearing in $(P)$ and $(P^d)$.  Given an 
uniform grid of size $h$ on the torus $\mathbb{T}^2$, we call 
$(P_h)$ and $(P_h^d)$ the chosen discrete versions of $(P)$ and 
$(P^d)$, respectively.  We prove the existence of at least one 
solution $(m^h,w^h)$ of the discrete  variational problem 
$(P_{h})$, as well as the existence of Lagrange multipliers 
$(u^h,\lambda^h)$ associated to $(m^h,w^h)$. Similar results are 
obtained for $(P_h^d)$, where an additional Lagrange multiplier 
$p^h$ appears because of the supplementary density constraint. 
We state the general optimality conditions for both problems in 
Theorem \ref{condicionesoptimalidad}. If we consider problem 
$(P_h)$ and we suppose that  $\nu>0$, we obtain in Corollary 
\ref{masaspositivas} that $(m^h,u^h,\lambda^h)$ solves the 
finite-difference scheme  proposed by Achdou et al. in  
\cite{AchdouCapuzzo10,AchdouCapuzzoCamilli10}.  We point out 
that, contrary to \cite{AchdouCapuzzoCamilli12}, our analysis 
does not use convex duality theory and thus allows, at this 
stage, to consider non-convex functions $F$ in order to obtain 
the existence of solutions to the discrete systems, recovering   
some of the results in \cite{AchdouCapuzzo10}, without using 
fixed point theorems.  When $\nu=0$,  we obtain in Corollary 
\ref{casomasagennuigual0} the existence of a solution of a 
natural discretization of the stationary first order MFG system 
proposed in  \cite[Definition 4.1]{Cardagraber15}. Analogous 
existence results, based on the study of problem $(P^d_h)$, are 
proved for   natural discretizations of   system 
\eqref{mfgestacionarioconrestcaja}.

If  $\nu>0$, $H$ is of the form \eqref{hstarpowerq}, $f(x,\cdot)$ is 
increasing,  and we suppose that   \eqref{mfgestacionario} admits  
regular solutions, then, as $h\downarrow 0$, the sequence of 
solutions  $(m^h,u^h,\lambda^h)$  of the finite-difference scheme 
proposed in \cite{AchdouCapuzzo10}  converges to the unique solution 
of \eqref{mfgestacionario} (see  \cite[Theorem 
5.3]{AchdouCapuzzoCamilli12}). One can then use Newton's method to 
compute $(m^h,u^h,\lambda^h)$ (see \cite{AchdouCapuzzo10}, where the 
stationary solution is approximated with the help of  time-dependent 
problems, and \cite{CaCa16}, where a direct approach is used) and so 
the computation is efficient  if the initial guess for   Newton's 
algorithm  is near to the solution. On the other hand, as pointed out 
in \cite[Section 5.5]{AchdouCapuzzoCamilli10}, \cite[Section 
2.2]{MR2928376} and \cite[Section 9]{CaCa16} the performance of 
Newton's method heavily depends on the values of $\nu$: for small 
values, or in the limit case when $\nu=0$, the convergence is much 
slower and, numerically and without suitable modifications, cannot be 
guaranteed because  the iterates for the computation of $m^h$ can 
become negative. 

If $f$ is increasing with respect to its second argument, then problems $(P)$ and $(P^d)$ are convex, a property that is preserved by the discrete versions $(P_h)$ and $(P_h^d)$. Therefore, it is natural to consider first order convex optimization algorithms (see \cite{BausCombettes11} for a rather complete account of these techniques) to overcome the difficulties explained in the previous paragraph. In particular, these algorithms are {\it global} because they converge for any initial condition.  This type of strategy has been already pursued in the articles \cite{benamoucarlier15,BenCarSan16,Andreev16} where the Alternating Direction Method of Multipliers (ADMM), introduced in \cite{MR0388811,GABAY197617,Gabay1983299}, is applied to solve some MFG systems.  The ADMM method is a variation of the well-known Augmented Lagrangian method, introduced in  \cite{MR0262903,MR0271809,MR0272403}, and has been successfully applied in the context of optimal transportation problems (see e.g. \cite{MR1738163,MR2516195}). This method shows good performance in the case when $\nu=0$ (see \cite{benamoucarlier15,BenCarSan16})  and  has been recently tested when $\nu >0$ and the MFG model is time-dependent (see \cite{Andreev16}, where some preconditioners are introduced in order to solve the linear systems appearing in the iterations). We also mention \cite{gomesex}, where the monotonicity of $f$ also plays an important role in order to obtain the convergence of the flows constructed to approximate the solutions. Finally, we refer the reader to the articles \cite{aime10,burgerwolfram} for some numerical methods to solve some non-convex variational MFG.

In this work we study the applicability of several first order proximal methods to solve both problems $(P_h)$ and $(P_h^d)$ with $\nu\geq 0$ being a small,  possibly null,  parameter.  In order to implement these types of methods,  in Section \ref{proxvarphisection} we  compute efficiently the proximity operators of  the cost functionals appearing in $(P_h)$ and $(P_h^d)$. We consider and compare the Predictor-Corrector Proximal Multiplier (PCPM) method proposed by Chen  and Teboulle in \cite{ChenTeb94}, a proximal method based on the splitting of a Monotone  plus Skew (MS) operator, introduced by   Brice\~no-Arias and Combettes in \cite{Siopt1},  and a   primal-dual method proposed by Chambolle and Pock (CP) in  \cite{CPock11}. Depending on whether we split or not the influence of the linear constraints in $(P_h)$ and $(P_h^d)$ we get two different implementations of each algorithm. Loosely speaking,  if we split the operators we increase the number of explicit steps per iteration but we do not need to invert matrices, which sometimes can be costly or even prohibitive. We have observed numerically that methods with splitting can be accelerated by projecting the iterates into some of the constraints. It can be proved that this modification does not alter the convergence of the method (see the Appendix for a proof of this fact in the case of the algorithm by CP). When $\nu=0$, we compare all the three methods in a particular instance of problem $(P)$, taken from \cite{gomesex}, which admits an explicit solution.  All the methods achieve a first-order convergence rate and we observe that  the algorithm CP is the one that performs better.  Next, for an example taken from \cite{benamoucarlier15}, we compare the performances and accuracies of the algorithms CP and  ADMM. We find in this example that for low and zero viscosities the algorithm CP obtains the same accuracy than the ADMM method but with fewer iterations. The situation changes for higher viscosities where we observe faster computation times for the ADMM method. Finally, we show that the method by CP also behaves very well when solving $(P^d_h)$, with computational times and numbers of iterations comparable to those  for  $(P_h)$.

The article is organized as follows: in the next section we set the 
notation that will be used throughout this paper, we recall the 
finite-difference scheme to solve $(P_h)$ proposed by Achdou and 
Capuzzo-Dolcetta in  \cite{AchdouCapuzzo10},
we define the discrete optimization problems studying their main 
properties, and we provide the optimality conditions at a solution 
$(m^h,w^h)$ (which is shown to exist). In particular, we obtain the 
existence of solutions of discrete versions of  
\eqref{mfgestacionario} and \eqref{mfgestacionarioconrestcaja}. In 
Section \ref{als}, we present a short survey of the proximal methods 
considered in this article and we compute the proximity operators of 
the cost functionals appearing in $(P_h)$ and $(P_h^d)$. Finally, in 
Section \ref{sec:num}, we present numerical experiments assessing  
the 
performance  of the different methods in several situations (small or 
null viscosity parameters, density constrained problems, various 
values of $q$,  etc).


%

\section{Discrete MFG and finite-dimensional optimization 
problems}\label{problemadimfinita}

In this section we recall some notation and the finite difference 
approximation 
of the MFG system 
introduced in \cite{AchdouCapuzzo10}. Then we set and study the 
finite-dimensional versions of the optimization problems $(P)$ and 
$(P^d)$, which are called $(P_h)$ and $(P_h^d)$, and we derive 
existence of solutions and their optimality conditions.
\subsection{Finite difference scheme}
 Following   \cite{AchdouCapuzzo10}, we consider 
an uniform grid $\TT_{h}^{2}$ on the two dimensional torus $\TT^{2}$ 
with step size $h>0$  such that 
$N_h:= 1/h$ is an integer. For a given function $y: \TT_{h}^{2}\to 
\RR$ and $0\leq i,j\leq N_h-1$, we 
set $y_{i.j}:= y(x_{i,j})$ (and thus we identify the set of functions 
$y:\TT_{h}^{2}\to \RR$ with $\RR^{N_h\times N_h}$)  and 
$$\ba{c} (D_{1}y)_{i,j}:= \frac{y_{i+1,j}- y_{i,j}}{h}, 
\hspace{0.6cm}  
(D_{2} y)_{i,j}:= \frac{y_{i,j+1}- y_{i,j}}{h}, \\[4pt]
		[D_{h}y]_{i,j}:= \left( (D_{1}  y)_{i,j}, (D_{1}  y)_{i-1,j}, 
		(D_{2}  y)_{i,j}, (D_{2} y)_{i,j-1}\right).\ea $$
The discrete Laplace operator $\Delta_{h} y: \TT_{h}^{2} \to \RR$ is 
defined by 
\be\label{defi:discretelaplaceop} (\Delta_{h} y)_{i,j}:= - 
\frac{1}{h^{2}}\left( 4 y_{i,j}- y_{i+1,j}-y_{i-1,j}- y_{i,j+1}- 
y_{i,j-1}\right).\ee
Given  $a\in \RR$,   set $a^{+}:=\max\{a,0\}$ and  $a^{-}:=a^{+}-a$, 
and define
\be\label{definicionnablahat} \widehat{[D_h y]}_{i,j}=\left( (D_{1} 
y)_{i,j}^{-}, -(D_{1} y)_{i-1,j}^{+},  (D_{2}  y)_{i,j}^{-}, - 
(D_{2}  y)_{i,j-1}^{+}\right).\ee
When $\nu>0$, the  Godunov-type finite-difference scheme proposed  
in  
\cite{AchdouCapuzzo10} to solve \eqref{mfgestacionario} reads as 
follows: Find $u^{h}$, $m^h 
:\TT_{h}^{2}\to \RR$ and $\lambda^{h} \in \RR$ such that, for all 
$0\leq i,j\leq N_h-1$,
\begin{equation}
\label{e:optdiscret}
	\ba{rcl} -\nu (\Delta_h u^h)_{i,j}+   \frac{1}{q'}| \widehat{[D_h 
		u^h]}_{i,j}|^{q'}-\lambda^{h} &=& f(x_{i,j},m^{h}_{i,j}), 
		\\[4pt]
	- \nu (\Delta_h m^{h})_{i,j}- \T_{i,j}(u^{h},m^{h})&=&0,\\[4pt]
	m_{i,j}^{h} \geq 0, \; \; \;  \sum_{i,j} u_{i,j}^{h}=0, & \; & 
	h^{2} \sum_{i,j} m_{i,j}^{h}=1,\ea
\end{equation}
where, for every $u'$, $m': \TT_{h}^{2}\to  \RR$ we set
\be\label{tij}\ba{rcl} h\T_{i,j}(u',m')&:=& -m_{i,j}'| 
\widehat{[D_h u']}_{i,j}|^{\frac{2-q}{q-1}} (D_{1}  
u')_{i,j}^{-}+m_{i-1,j}'| \widehat{[D_h 
u']}_{i-1,j}|^{\frac{2-q}{q-1}} 
(D_{1}  u')_{i-1,j}^{-} \\[4pt]
			\;           & \; & +m_{i+1,j}'| \widehat{[D_h 
			u']}_{i+1,j}|^{\frac{2-q}{q-1}} (D_{1}  u')_{i,j}^{+}	
			-m_{i,j}'| \widehat{[D_h u']}_{i,j}|^{\frac{2-q}{q-1}} 
			(D_{1}  u')_{i-1,j}^{+} \\[4pt]
			\; 	    & \; & -m_{i,j}'	| \widehat{[D_h 
			u']}_{i,j}|^{\frac{2-q}{q-1}} (D_{2}  u')_{i,j}^{-}	
			+m_{i,j-1}'| \widehat{[D_h 
			u']}_{i,j-1}|^{\frac{2-q}{q-1}} 
			(D_{2}  u')_{i,j-1}^{-}\\[4pt]
			\;  & \; &  +m_{i,j+1}'| \widehat{[D_h 
			u']}_{i,j+1}|^{\frac{2-q}{q-1}} (D_{2}  u')_{i,j}^{+}	
			-m_{i,j}'| \widehat{[D_h u']}_{i,j}|^{\frac{2-q}{q-1}} 
			(D_{2} 
			u')_{i,j-1}^{+}.										
				
			\ea\ee
%
%
As in the continuous case, we use the convention   that 
\be\label{convencion}| \widehat{[D_h 
u']}_{i,j}|^{\frac{2-q}{q-1}}\widehat{[D_h u']}_{i,j}=0 
\hspace{0.4cm} 
\mbox{if } \; q>2 \hspace{0.4cm} \mbox{and } \widehat{[D_h 
u']}_{i,j}=0,\ee 
which implies that $\T_{i,j}$ is well defined. Existence of a 
solution $(m^h,u^h,\lambda^h)$  to \eqref{e:optdiscret}  is proved 
in  \cite[Proposition 4 and Proposition 5]{AchdouCapuzzo10} using 
Brouwer's fixed point theorem. Several other features as stability  
and  robustness are also established in  \cite{AchdouCapuzzo10}. If 
$f$ is strictly increasing as a function of its second argument,   
uniqueness of a solution to \eqref{e:optdiscret} is proved in 
\cite[Corollary 1]{AchdouCapuzzo10}, and convergence to the solution 
to \eqref{mfgestacionario} when $h\downarrow 0$ is proven in 
\cite[Theorem 5.3]{AchdouCapuzzoCamilli12}, assuming that the latter 
system admits a unique smooth solution. Finally, we also refer the 
reader to \cite{doi:10.1137/15M1015455} for some convergence results 
of the analogous scheme in the framework of weak solutions for 
time-dependent MFG.

In the remaining of this section, we recover the existence of a 
solution 
to \eqref{e:optdiscret} from a purely 
variational approach. We will also prove the existence of solutions 
to the analogous discretization schemes for system 
\eqref{mfgestacionario} when $\nu=0$ and for system 
\eqref{mfgestacionarioconrestcaja} when $\nu \in [0,+\infty[$. First 
we introduce the associated finite-dimensional optimization problems.
\subsection{Finite-dimenstional optimization}
Inspired by \cite{AchdouCapuzzoCamilli10},    in the context  of the planning problem for MFG,   we introduce in this Section some finite dimensional analogues of the optimization problems $(P)$ and $(P^{d})$ and we study the existence of solutions as well as first-order optimality conditions. We introduce the following notation. Denote by $\RR_{+}$  the set of non-negative real numbers, by $\RR_{-}:= \left(\RR\setminus \RR_{+}\right)\cup \{0\}$,  let 
$K:=\RR_+\times\RR_-\times\RR_+\times\RR_-$, let $q\in ]1,+\infty[$, and  define  
\begin{equation}\label{definicionhatb} \hat{b}\colon\RR\times\RR^4\to\RX\colon(m,w)\mapsto 
\begin{cases}
\displaystyle{\frac{|w|^q}{q\, m^{q-1}}},\quad&\text{if }m>0,\,w\in K,\\
0,&\text{if }(m,w)=(0,0),\\
\pinf,&\text{otherwise}.
\end{cases}
\end{equation}
Let $\M_h:= \RR^{N_h \times N_h}$, $\W_h:= ( \RR^{4})^{N_h \times 
N_h}$ and let $d \in \M_h$ defined as $d_{i,j}:= d(x_{i,j})$, where $d$ is defined in $(P)$. Note that for $h$ small enough, we have that  $ h^2\sum_{i,j} d_{i,j} >1$.  Consider the mappings $A: \M_h \to \M_h $, $B: \W_h \to 
\M_h$ defined as 
$$\ba{c} (A m)_{i.j}=   -\nu(\Delta_h m)_{i,j}, \; \; \; (Bw)_{i, j} 
:=(D_{1}  w^{1})_{i-1,j}+ (D_{1}  w^{2})_{i,j}+ (D_{2} 
w^{3})_{i,j-1}+ (D_{2}  w^{4})_{i,j}.\ea$$
It is easy to check (see e.g. \cite{AchdouCapuzzo10}) that the 
adjoint mappings $A^{*}$ and $B^{*}$ satisfy $(A^{*}y)_{i,j}=- \nu(\Delta_h 
y)_{i,j}$ (i.e. $A$ is symmetric) and $(B^{*}y)_{i,j}= -[D_h y]_{i,j}$ for 
all $y \in \M_h$. In particular,  $\mbox{Im}(B)= \Y_h$ where 
\be\label{ygriegah}  \Y_h:= \left\{ y \in \M_{h} \; ; \; \sum_{i,j} y_{i,j}=0\right\}.\ee
 Indeed, 
note that if $y\in \Y_h$ satisfies $-[D_h y]_{i,j}=(B^{*}y)_{i,j}=0$ for 
all $(i,j)$ then $y$ must be constant and so, since $\sum_{i,j} y_{i,j}=0$, 
we must have that  $y=0$.

Now, recalling the definition of $F$ in \eqref{e:Lq}, define $\B: \M_{h}\times \W_{h} \mapsto \RR \cup \{+\infty\}$ and $\F: \M_{h}\mapsto \RR \cup \{+\infty\}$ as 
\be\label{definicionfuncompletas} \B(m,w)= \sum_{i,j} 
\hat{b}(m_{i,j},w_{i,j})  \hspace{0.3cm} \mbox{and }\hspace{0.2cm} 
\F(m):= \sum_{i,j} F(x_{i,j}, m_{i,j}).\ee
In addition, define the function  $G:   \M_h\times  \W_h \mapsto \M_h\times \RR $  and the closed and convex  set $\D$   as 
\be\label{defGyD}\ba{c}
G(m,w):= (Am+Bw,  h^{2} \sum_{i,j} m_{i,j}), \\[4pt]
\D:= \{ (m',w') \in \M_h \times \W_h \; ; \; m'_{i,j} \leq d_{i,j} \hspace{0.3cm} \mbox{for all $i$, $j$}\}.
\ea\ee
In this work we consider the following discretization of $(P^{d})$ 
$$\inf_{(m,w) \in \M_{h} \times \W_h}  \B(m,w)+ \F(m)  \hspace{0.5cm}
\mbox{s.t. }  \hspace{0.3cm}  G(m,w)=(0,1) \in \M_{h}\times \RR, \; \; \; m \in \D, 
  \eqno(P_h^{d})$$
and the corresponding discretization of $(P)$
$$\inf_{(m,w) \in \M_{h} \times \W_h}  \B(m,w)+ \F(m)  \hspace{0.5cm}
\mbox{s.t. }  \hspace{0.3cm}  G(m,w)=(0,1) \in \M_{h}\times \RR. 
  \eqno(P_{h})$$

\subsection{Existence and optimality conditions of the discrete 
problems $(P_{h})$ and $(P_h^{d})$}

In order to derive necessary conditions for optimality in problems 
$(P_h^{d})$ and $(P_h)$ we need the computation of  $\hat{b}^{\ast}$  
and $\partial \hat{b}^{\ast}$, where $\hat{b}$ is defined in 
\eqref{definicionhatb}. Recall 
that, given a subset $C\subset\RR^{n}$,  $\iota_C$ is defined as 
$\iota_C(c)=0$ if $c\in C$ and $+\infty$ otherwise. If $C$ is 
non-empty,  
closed, and convex, the normal cone to $C$ 
at $x\in C$ is defined by 
$$ N_{C}(x):= \left\{ y \in \RR^{n} \; ; \;  y\cdot (c-x)\leq 0, 
\; \; \forall \; c\in C\right\}.$$
 If $C$ is a cone, we will denote by $C^{-}$ its polar cone, 
 defined 
 as 
 $C^{-}:=\left\{c^{*} \in \RR^{n} \; ; \;   c^{*}\cdot c  \leq 
 0, 
 \; \forall c\in C\right\}.$
 We also recall that for a given  a proper lower 
  semi-continuous (l.s.c.) convex function 
 $\ell: \RR^{n} \mapsto ]-\infty,+\infty]$, the Fenchel 
 conjugate $\ell^{\ast}: 
 \RR^{n} \mapsto ]-\infty, +\infty]$ is defined by
$$\ell^{\ast}(p):= \sup_{x\in \RR^{n}} \left\{   p\cdot x - 
\ell(x)\right\} $$
and the subdifferential $\partial \ell(x)$ of $\ell$ at $x$ is 
defined as 
the set of points $p\in \RR^{n}$ such that 
\be\label{subdiferencialdef}\ell(x)+   p\cdot(y-x)\leq \ell (y) 
\hspace{0.3cm} \forall 
\; y\in \RR^{n}.\ee
A useful characterization of the subdifferential states that at 
every $x\in \mbox{dom}(\ell):=\{ x\in \RR^{n} \; ; \; 
\ell(x)<\infty\}$, we have
\be\label{subdifferentiacharacterization}\partial \ell(x) = 
\mbox{argmax}_{p\in \RR^{n}} \left\{ p\cdot x- 
\ell^{\ast}(p)\right\}.\ee
For $x 
\in \RR^4$ we set  $P_Kx$ for its projection into the set $K$.
%
\begin{lemma}
\label{l:partial} 
The function $\hat{b}$ is proper, convex, and l.s.c. Moreover,  setting   
$$C:=\Menge{(\alpha,\beta)\in\RR\times\RR^4}{\alpha+
\displaystyle{\frac{1}{q'}|P_K \beta|^{q'}\leq0}}$$
we have that $\hat{b}^{\ast}=\iota_{C}$ and 
\begin{equation}\label{subdifferentialhatb}
\partial\hat{b}\colon(m,w)\mapsto
\begin{cases}
\displaystyle{\Big(-
\frac{1}{q'}\frac{|w|^q}{m^q},
\frac{|w|^{q-2}}{m^{q-1}}w\Big)}+\{0\}\times N_K(w),\quad&\text{if
 }m>0,\\
C,&\text{if 
}(m,w)=(0,0),\\
\emptyset,&\text{{\rm otherwise}}.\end{cases}
\end{equation}
\end{lemma}
\begin{proof} Note that   $ \hat{b}(m,w)=b(m,w)+\iota_{\RR\times 
K}(m,w)$, where $b: \RR \times \RR^4 \to [0,+\infty]$ is defined 
by 
$$\ba{l}
b(m,w):= \left\{
\begin{array}{ll}
\frac{|w|^q}{q\, m^{q-1}} & \mbox{if } m>0,\\
0, & \mbox{if } (m,w)=(0,0),\\
+\infty, & \mbox{otherwise,}
\end{array}
\right. \hspace{0.3cm}
\ea
$$
or equivalently (see  e.g. \cite{santambrogiobook}),
\be\label{bqdualrepresentation}
b(m,w)= \sup_{(\alpha,\beta)\in E}\left\{\alpha 
m+\beta \cdot w\right\}, \; \; \mbox{where } \;  
E:=\left\{(\alpha,\beta)\in\RR\times\RR^4 \; ; \;
 \alpha+\frac{1}{q'}|\beta|^{q'}\leq 0\right\}.
\ee
Since  $\RR\times K$  is convex, closed and non-empty, we have that $b$ and  $\hat{b}$ are proper, convex and l.s.c. Moreover, \eqref{bqdualrepresentation} implies that $b^{\ast}=\iota_{E}$.  In order to compute $\hat{b}^{\ast}$ and $\partial \hat{b}$ we first prove that $C=E+\{0\}\times K^-$. Indeed,  every $\beta \in \RR^{4}$ can be written as 
$\beta= P_{K}(\beta)+ P_{K^{-}}(\beta)$ from which the inclusion  
$C\subseteq E+\{0\}\times K^-$ follows. Conversely, for any $\beta \in \RR^4$ 
and $n\in K^{-}$ we have that 
$|P_{K}(\beta+n)|=|P_{K}(\beta+n)-P_{K}(n)|\leq |\beta|$ and so we get $  
E+\{0\}\times K^- \subseteq C$. Now, 
using the identity
$\hat{b} =b +\iota_{\RR\times K}$, the fact that    $b$ is finite and 
continuous at $(1,0)\in \M_{h} \times \W_{h}$ and that $ \iota_{\RR\times 
K}(1,0)=0$,  by \cite[Theorem 9.4.1]{attouch} we have that 
\be\label{bqast1}\ba{rcl}\hat{b}^{\ast}(\alpha,\beta)&=& \inf_{(\alpha',\beta') \in \RR \times 
\RR^{4}}\left\{ b^{\ast}(\alpha-\alpha', \beta-\beta')+ \iota_{\RR\times 
K}^{\ast}(\alpha',\beta')\right\} \\[4pt]
\; &=&\inf_{(\alpha',\beta') \in \RR \times 
\RR^{4}}\left\{ \iota_{E}(\alpha-\alpha', \beta-\beta')+ 
\iota_{\RR\times K}^{\ast}(\alpha',\beta')\right\}.\ea\ee
 It is easy to see that $\iota_{\RR\times 
K}^{\ast}(\alpha',\beta')=\iota_{\{0\}\times 
K^{-}}(\alpha',\beta')$. Using that $C=E+\{0\}\times K^-$, from 
\eqref{bqast1} we obtain  
 $\hat{b}^{\ast}=\iota_{C}$. 
 
Let us now prove \eqref{subdifferentialhatb}.  Since $\hat{b} =b +\iota_{\RR\times K}$, it follows from  
\cite[Chapter 1, Proposition 5.6]{MR0463993}  that 
for every $(m,w)\in\RR\times\RR^4$,
$\partial\hat{b}(m,w)=\partial b(m,w)+N_{\RR\times 
K}(m,w)=\p b(m,w)+\{0\}\times N_K(w)$.  Now, using \eqref{subdifferentiacharacterization} and $b^{\ast}=\iota_{E}$ we get
\be\label{maximizationonE}\p b(m,w)=\Argmax{(\alpha,\beta)\in 
E}\left\{\alpha m+\beta \cdot w\right\},\ee
from which we readily obtain that $\partial b(m,w)=\emptyset$ 
if $m<0$ and if $m=0$ and $w\neq0$. Thus, the third case in 
\eqref{subdifferentialhatb} follows.  If $m>0$,  then $b$ is differentiable and so 
$$\p 
b(m,w)=\left\{\left(-\frac{1}{q'}\frac{|w|^{q}}{m^q},
\frac{|w|^{q-2}}{m^{q-1}}w\right)\right\},$$
from which the first case in \eqref{subdifferentialhatb} follows.
Finally, if $(m, w)=(0,0)$ using \eqref{maximizationonE} we get 
that $\partial b(0,0)=E$. On the other hand, note  that $N_K(0)=K^-$ 
and  so $\p \hat{b}(0,0)=C$. The result follows.\end{proof}\medskip

In the following result we prove a {\it qualification condition}, which will be useful for establishing optimality conditions.
\begin{lemma}\label{separacionsubdiferencial}  There exists $(\tilde{m}, \tilde{w}) \in \M_{h} \times \W_h$ such that 
\be\label{qualificationcondition}
G(\tilde{m},\tilde{w})=(0,1), \hspace{0.5cm} \tilde{w}\in \mbox{{\rm int}}(K), \hspace{0.5cm} 0<\tilde{m}_{i,j}<d_{i,j} \; \; \forall \; (i,j).
\ee
\end{lemma}
\begin{proof} Since $d_{i,j}>0$ for all $(i,j)$ and $h^{2}\sum_{i,j} d_{i,j} > 1$, there exists  $\tilde{m} \in \M_{h}$ satisfying   $0<\tilde{m}_{i,j}<d_{i,j}$  for all $(i,j)$ and $h^{2}\sum_{i,j} \tilde{m}_{i,j}=1$. Since $ A\tilde{m} \in \Y_h$ (recall \eqref{ygriegah}) and $\mbox{Im}(B)= \Y_h$,  there exists $\hat{w}\in \W_h$ such that $A \tilde{m}+ B\hat{w}=0$. Given $\delta>0$ and letting $\tilde{w}:= \hat{w}+ c_{\delta}$ with 
$$(c_{\delta})_{i,j}:= \left(\max_{i,j}|w_{i,j}^1|+ \delta, \;  -\max_{i,j}|w_{i,j}^2|- \delta, \; \max_{i,j}|w_{i,j}^3|+ \delta, \;  -\max_{i,j}|w_{i,j}^4|- \delta\right) \hspace{0.4cm} \mbox{for all $i$, $j$},$$ 
we have that $\tilde{w} \in \mbox{int}(K)$ and  $G(\tilde{m}, \tilde{w})=(0,1)$. 
\end{proof}\medskip


Now, we prove the main result of this Section. 
\begin{theorem}\label{condicionesoptimalidad} 
For any $\nu \geq 0$ the following assertions hold true: \smallskip\\
{\rm(i)}  Problems $(P_{h}^{d})$ and $(P_{h})$ admit  at least one solution and the optimal costs are finite. \smallskip\\
{\rm(ii)} Let $(m^h, w^h)$ be a solution to $(P_{h}^{d})$. Then, there exists $(u^h, \mu^h, p^h, \lambda^h) \in (\M_h)^{3} \times \RR$ such 
that \small
\be\label{mfgdiscretogeneral}\ba{rcl} -\nu(\Delta_h u^h)_{i,j}+   
\frac{1}{q'}| \widehat{[D_h 
u^h]}_{i,j}|^{q'}+\mu_{i,j}^h-p_{i,j}^{h}-\lambda^{h} &=& f(x_{i,j},m_{i,j}^{h}), 
\\[4pt]
		- \nu(\Delta_h m^{h})_{i,j}- \T_{i,j}(u^{h},m^{h})&=&0,\\[4pt]
		\sum_{i,j} u_{i,j}^{h}=0, & \; &0\leq m_{i,j}^{h} \leq d_{i,j}, \; \;  h^{2} \sum_{i,j} m_{i,j}^{h}=1, \\[4pt]
		\mu_{i,j}^{h}\geq 0, \; \; p_{i,j}^{h} \geq 0, &\;&  m_{i,j}^{h}  \mu_{i,j}^{h}=0, \; \; (d_{i,j}-m_{i,j}^{h})p_{i,j}^{h}=0.
		\ea \ee \normalsize
{\rm(iii)} Let $(m^h, w^h)$ be a solution to $(P_{h})$. Then, there exists $(u^h, \mu^h, \lambda^h) \in (\M_h)^{2} \times \RR$ such 
that \small
\be\label{mfgdiscretogeneralsinrestricciondedensidad}\ba{rcl} -\nu(\Delta_h u^h)_{i,j}+   
\frac{1}{q'}| \widehat{[D_h 
u^h]}_{i,j}|^{q'}+\mu_{i,j}^h-\lambda^{h} &=& f(x_{i,j},m_{i,j}^{h}), 
\\[4pt]
		- \nu(\Delta_h m^{h})_{i,j}- \T_{i,j}(u^{h},m^{h})&=&0,\\[4pt]
		 \sum_{i,j} u_{i,j}^{h}=0. &\;&  m_{i,j}^{h}\geq 0, \; \; h^{2} \sum_{i,j} m_{i,j}^{h}=1,\\[4pt]
		\mu_{i,j}^{h}\geq 0,   & \; &  m_{i,j}^{h}  \mu_{i,j}^{h}=0.
		\ea \ee \normalsize

\end{theorem}
\begin{proof} We only prove {\rm(i)} and {\rm(ii)} since the proof of {\rm(iii)} is analogous to that of {\rm(ii)}.\\
 {\it Proof of assertion {\rm (i)}: }  
Lemma \ref{separacionsubdiferencial}  implies the existence of   $(\tilde{m}, \tilde{w})$ feasible for both problems and having a finite cost. 
In order to prove the existence of an optimum for $(P_h^{d})$ or for $(P_h)$, note that since 
%
$\hat{b}(m_{i.j},w_{i,j})= \hat{b}(m_{i,j},w_{i,j})+ 
\iota_{\RR_{+}}(m_{i,j})$ and that we have the constraint   $h^{2} 
\sum_{i,j} m_{i,j}=1$,   any   minimizing sequence $(m^n,w^n)$   satisfying  that $|(m^n,w^n)|\to \infty$ must 
satisfy that, except for some subsequence,  there exists $(i,j)$ such that 
$|w^{n}_{i,j}| \to \infty$.  Independently of the value of $m_{i,j}^{n}$, 
we obtain  $\hat{b}(m_{i,j}^{n}, w_{i,j}^{n}) \to \infty$.  Since the continuity of $F$ and boundedness of $m^n$ imply that $F(x_{i,j},m_{i,j}^{n})$ is uniformly bounded in $n$ and $(i,j)$, we get that 
$\sum_{i,j}\left[ \hat{b}(m_{i,j}^{n}, w_{i,j}^{n})+ 
F(x_{i,j},m_{i,j}^{n})\right] \to \infty$, which implies that any minimizing sequence must be bounded. Since, in addition, the cost is l.s.c., we obtain that, 
independently of the value of $\nu \geq 0$, problems $(P_h^{d})$ and $(P_h)$   admit at 
least one solution $(m^{h},w^{h})$. \smallskip\\
 {\it Proof of assertion {\rm (ii)}: } Recalling \eqref{definicionfuncompletas} and \eqref{defGyD}, problem $(P_h^{d})$ can be written as
$$\inf_{(m,w) \in \M_{h} \times \W_h} \B(m,w) +\F(m)+ 
\iota_{G^{-1}(0,1)}(m,w)+   \iota_{\D}(m).$$
For $m\in \M_h$ such that $m_{i,j}\geq 0$ for all $i,j$, we set $ (\nabla \F^+(m))_{i,j}:= f(x_{i,j},m_{i,j})\in \RR$ for all $i,j$.
Let us prove that at the optimum $(m^h,w^h)$ 
\be\label{nablain}(-\nabla \F^{+}(m^h),0) \in  \partial_{(m,w)} E(m^h,w^h), \hspace{0.3cm} \mbox{where } \; E:=\B+\iota_{G^{-1}(0,1)}+ \iota_{\D}.  \ee
Indeed, by optimality, for each $(m',w')$ such that $(m')_{i,j} \geq 0 $ for all $(i,j)$,
 we have 
$$ \F(m^h)+ E(m^h, w^h) \leq \F(m^{h}+ \tau (m'-m^h)) + E(m^{h}+ \tau (m'-m^h),w^{h}+ \tau (w'-w^h)),$$
for every $\tau \in ]0,1]$, and so \small
\be\label{diferenciasdeFe}
-\frac{1}{\tau} \left(\F(m^{h}+ \tau (m'-m^h))- \F(m^h)\right) \leq \frac{1}{\tau}\left(E(m^{h}+ \tau (m'-m^h),w^{h}+ \tau (w'-w^h))- E(m^h, w^h) \right).
\ee \normalsize
Using the convexity of $E$, the right-hand-side   of \eqref{diferenciasdeFe}  is bounded by its value at $\tau=1$, i.e. $E(m',w')-E(m^h, w^h)$.  On the other hand, the continuity of $f$ implies that the left-hand-side  of \eqref{diferenciasdeFe}  converges to $ -\nabla \F^{+}(m^h) \cdot (m'-m^h)$ as $\tau \downarrow 0$  and, hence,
\be\label{subgradientE}  E(m^h, w^h)-\nabla \F^{+}(m^h) \cdot (m'-m^h) \leq  E(m',w').\ee
Now, if $(m')_{i,j} <0$ 
for some $(i,j)$  then the right hand side of  \eqref{subgradientE} is $+\infty$ and the inequality is trivially verified.  
Relation \eqref{nablain} follows from the definition \eqref{subdiferencialdef} and \eqref{subgradientE}. 

Now, let $(\tilde{m},\tilde{w})$ satisfying \eqref{qualificationcondition} in Lemma \ref{separacionsubdiferencial}. Since $(m,w) \mapsto \B(m,w)$ is finite and continuous at 
$(\tilde{m}, \tilde{w})$ and $\left(\iota_{G^{-1}(0,1)}+\iota_{\D}\right)(\tilde{m}, \tilde{w}) =0$, by 
\cite[Chapter 1, Proposition 5.6]{MR0463993} at the  optimum $(m^h,w^h)$  
we have 
\be\label{ecuaciondeoptimalidad}(-\nabla \F^{+}(m^h),0) \in \partial_{(m,w)} \B(m^{h},w^{h})  + \partial_{(m,w)} \left(\iota_{G^{-1}(0,1)}+\iota_{\D}\right)(m^h, w^h).\ee
Using that $\iota_{G^{-1}(0,1)}$ is finite at $(\tilde{m},\tilde{w})$ and that $\iota_{\D}$ is continuous at $(\tilde{m},\tilde{w})$, we obtain
$$\ba{rcl}  \partial_{(m,w)} \left(\iota_{G^{-1}(0,1)}+\iota_{\D}\right)(m^h, w^h)&=&   \partial_{(m,w)} \iota_{G^{-1}(0,1)}(m^h,w^h)+\partial_{(m,w)}\iota_{\D}(m^h,w^h),\\[4pt]
\; &=&  N_{G^{-1}(0,1)}(m^h,w^h)+ N_{\D}(m^h,w^h).\ea$$
Clearly, \small
\be\label{cononormalG}
\ba{rcl} N_{G^{-1}(0,1)}(m^h,w^h)&=& \left\{ (-A^{*} u +\lambda {\bf 
1}_{\M_h}, -B^{*}u) \; ; \; u \in \M_h, \; \; \lambda \in \RR\right\},\\[4pt]
N_{\D}(m^h,w^h)&=&\left\{ p \in \M_{h} \; ; p_{i,j} \geq 0, \; \; p_{i,j}(d_{i,j}-m_{i,j}^{h})=0 \; \mbox{for all $i$, $j$}\right\}\times \{0\},
\ea
\ee \normalsize
where $( {\bf 1}_{\M_h})_{i,j}=1$ for all $i$, $j$. 
Using that $(A^{*}u)_{i,j}= -\nu(\Delta_{h}u)_{i,j}$ and 
$(B^{*}u)_{i,j}=-[D_hu]_{i,j}$, relations 
\eqref{ecuaciondeoptimalidad}-\eqref{cononormalG} yield  the existence of 
$u^{h} \in \M_h$, $p^h   \in \M_h$ such that $(p^h,0)\in  N_{\D}(m^h,w^h)$, and   $\lambda^{h} \in \RR$ such that 
\be\label{inclusionsubdiferencial}
\left(-\nu(\Delta_{h}u^{h})_{i,j} -p^{h}_{i,j}- \lambda^{h}-f(x_{i,j},m_{i,j}^{h}),-[D_hu^{h}]_{i,j}\right)\in  \partial 
\hat{b}(m_{i,j}^{h} , w_{i,j}^{h}) \hspace{0.4cm} \forall i,j. \ee
If $m_{i,j}^{h}>0$, then  Lemma \ref{l:partial} yields 
\be\label{cond.optim.}\ba{rcl} -\nu(\Delta_{h} u^{h})_{i,j}+\frac{1}{q'} 
\frac{|w_{i,j}^{h}|^{q}}{(m_{i,j}^{h})^{q}} -p^{h}_{i,j}- \lambda^{h} &=& 
f(x_{i,j},m_{i,j}^{h}), \\[6pt]
		-[D_h u^{h}]_{i,j} &\in 
		&\frac{|w_{ij}^{h}|^{q-2}}{(m_{ij}^{h})^{q-1}} w_{ij}^{h}+  
		N_{K}(w_{ij}^{h}). \ea\ee
Using the last relation, if $w_{i,j}^{h}=0$ then $-[D_h u^{h}]_{i,j} \in 
N_{K}(0)= K^{-}$ and so $P_{K}(-[D_h u^{h}]_{i,j})=0$. Otherwise, we get 
\be\label{antesdedividir}|w_{i,j}^{h}|^{q-2}w_{i,j}^{h}= 
(m_{i,j}^{h})^{q-1}P_{K}(-[D_h u^{h}]_{i,j}),\ee
which is also valid for $w_{i,j}^{h}=0$.  Therefore, noting that $ 
\widehat{[D_h u^{h}]}_{i,j}=P_{K}(-[D_h u^{h}]_{i,j})$, using  
convention \eqref{convencion}, from \eqref{antesdedividir}   we deduce  
\be\label{expresionw} w_{i,j}^{h}= m_{i,j}^{h}  \left|  \widehat{[D_h 
u^{h}]}_{i,j} \right|^{\frac{2-q}{q-1}} \widehat{[D_h u^{h}]}_{i,j}\ee
and 
$$
\frac{|w_{i,j}^{h}|^{q}}{(m_{i,j}^{h})^{q}}= \left|\widehat{[D_h 
u^{h}]}_{i,j}\right|^{\frac{q}{q-1}}=  \left|\widehat{[D_h 
u^{h}]}_{i,j}\right|^{q'},
$$
which, together with the first equation in \eqref{cond.optim.},  yields the 
first equation in \eqref{mfgdiscretogeneral} with $\mu_{i,j}^{h}=0$. On the other hand, if 
$m_{i,j}^{h}=0$  then $w_{i,j}^{h}=0$ and, hence, relation \eqref{expresionw} is 
trivially satisfied (using   convention \eqref{convencion} again). Recalling the definition of $\T_{i.j}$ in \eqref{tij}, after some simple computations we deduce  that the second  equation  in \eqref{mfgdiscretogeneral} holds true in both cases 
($m_{i,j}^{h}=0$ and $m_{i,j}^{h}>0$). Now, if $m_{i,j}^h=0$, relation 
\eqref{inclusionsubdiferencial} and Lemma \ref{l:partial}   imply that
$$
-\nu(\Delta_{h}u^{h})_{i,j}+ \frac{1}{q'}| \widehat{[D_h 
u^{h}]}_{i,j}|^{q'} -p^{h}_{i,j} - \lambda^{h} \leq f(x_{i,j},0)
$$
Defining 
$$\mu_{i,j}^{h}=  
f(x_{i,j},0)+\nu(\Delta_{h}u^{h})_{i,j}-\frac{1}{q'}| 
\widehat{[D_h u^{h}]}_{i,j}|^{q'}+ p^{h}_{i,j}+\lambda^{h}$$
we get the the first equation in $(P_h)$ when $m_{i,j}^{h}=0$. Finally, by adding a constant we can always redefine $u^h$ in such a way such that $\sum_{i,j} u^{h}_{i,j}=0$.  The result  follows.  
\end{proof}\smallskip

The next result follows directly from Theorem 
\ref{condicionesoptimalidad}.  We write  the result explicitly  
only because of its analogy with the notion of weak solution in 
the  continuous case (see  \cite[Definition 4.1]{Cardagraber15} 
in the case without upper bound constraints for $m$). 

\begin{corollary}\label{casomasagennuigual0} In the case when 
$\nu=0$, for any solution $(m^{h}, w^{h})$ to $(P_h^{d})$  there 
exists $(u^h,p^{h},\lambda^h)\in \M_{h}^{2} \times \RR$ such 
that 
	\be\label{firsorderinequality1}\ba{rcl} 
	\frac{1}{q'}| \widehat{[D_h u^{h}]}_{i,j}|^{q'}-p_{i,j}^{h}-\lambda^{h} &\leq& f(x_{i,j},m_{i,j}^{h}), 
	\\[4pt]
	\T_{i,j}(u^{h},m^{h})&=&0,\\[4pt]
	\sum_{i,j} u_{i,j}^{h}=0, & \; & 0\leq m_{i,j}^{h} \leq d_{i,j}, \; \; \; h^{2} \sum_{i,j} m_{i,j}^{h}=1, \\[4pt]
	p_{i,j}^{h} \geq 0, &\;&   (d_{i,j}-m_{i,j}^{h})p_{i,j}^{h}=0.
	\ea \ee
	Similarly,  for any solution $(m^{h}, w^{h})$ to $(P_h)$  there exists $(u^h,\lambda^h)\in \M_h \times \RR$ such that
	\be\label{firsorderinequality2}\ba{rcl} 
	\frac{1}{q'}| \widehat{[D_h u^{h}]}_{i,j}|^{q'}-\lambda^{h} &\leq& f(x_{i,j},m_{i,j}^{h}), 
	\\[4pt]
	\T_{i,j}(u^{h},m^{h})&=&0,\\[4pt]
	\sum_{i,j} u_{i,j}^{h}=0, & \; & 0\leq m_{i,j}^{h}, \;  \; \; h^{2} \sum_{i,j} m_{i,j}^{h}=1.
	\ea \ee
	Moreover, in both systems, at each $i,j$ such that 
	$m_{i,j}^h>0$, we have that the first inequality is an 
	equality.
\end{corollary}  


\begin{remark}
The convergence when $h\to 0$ of solutions to 
\eqref{firsorderinequality1} and \eqref{firsorderinequality2} to
solutions to the correspondent continuous systems {\rm(}if they 
exist{\rm)} is out of the scope of this paper and remains as an 
interesting problem to 
be studied.
\end{remark}

We now drop the continuity 
assumption of $f$ in 
$\TT^{n}\times [0,\infty)$ by assuming that $f: \mathbb{T}^n 
\times ]0,+\infty[ \to \RR$ is  a continuous function such that 
$\int_{0}^{m} f(x,m') \dd m' \in \RR$ for all $x\in \TT^n$ and 
$m>0$. In the following result we prove the strict positivity of 
$m^h$ when $\nu>0$.
 In particular, it provides a variational proof 
 of the existence of a solution to the discrete MFG system in 
 the case of local interactions, first proved in 
 \cite{AchdouCapuzzo10} using the Brouwer's Fixed Point Theorem.

\smallskip

\begin{corollary}\label{masaspositivas} Suppose that $\nu>0$. 
Then, every solution $(m^{h},w^{h})$ to $(P_h^{d})$ or to 
$(P_h)$ satisfies that $m^{h}_{i,j}>0$ for all $i$, $j$. 
Consequently, systems \eqref{mfgdiscretogeneral} and 
\eqref{mfgdiscretogeneralsinrestricciondedensidad} are satisfied 
with $\mu_{i,j}=0$ for all $i$, $j$.
%
\end{corollary}
\begin{proof}  Let   $(m^{h},w^{h})$ be a solution to problem 
$(P_h^{d})$ or to  $(P_h)$  and suppose that there exists 
$(i,j)$ such that $m_{i,j}^{h}=0$. Then, since the cost function 
is finite at $(m^{h},w^{h})$, we must have that 
$w_{i,j}^{h}=0$.  Thus, the constraint $Am^h+ Bw^h=0$  implies 
that 
$$\frac{\nu}{h^{2}}\left(m_{i+1,j}^{h}+ 
m_{i-1,j}^{h}+m_{i,j+1}^{h}+m_{i,j-1}^{h}\right)=\frac{1}{h}\left(-(w^{h})_{i-1,j}^{1}
   + (w^{h})_{i+1,j}^{2} - 
(w^{h})_{i,j-1}^{3}+(w^{h})_{i,j+1}^{4}  \right).$$
Using again that the cost is finite at $(m^h, w^h)$, we must 
have that $w^{h}_{i',j'}\in K$ for all $(i',j')$, which implies 
that the right hand side in the above equation is non-positive. 
Since $m^h\geq 0$, we deduce 
$$0=m_{i+1,j}^{h}= m_{i-1,j}^{h}=m_{i,j+1}^{h}=m_{i,j-1}^{h}.$$
Reasoning recursively, we obtain $m^{h}=0$, 
contradicting that $h^{2}\sum_{i,j}m_{i,j}^{h}=1$. Therefore, we 
deduce that $m^h$ is strictly positive, and since $f$ is 
continuous in $\TT^{n}\times 
]0,+\infty[$, we obtain $(\nabla 
\F^{+}(m^h))_{i,j}=f(x_{i,j},m_{i,j}^{j})\in \RR$ for all $i,j$ 
and the proof in Theorem  \ref{condicionesoptimalidad}  can be 
reproduced analogously.
\end{proof}\smallskip

In general, if $\nu=0$ we cannot ensure the strict positivity of 
$m^{h}$ in any solution $(m^{h},w^{h})$ to $(P_h^{d})$ or to 
$(P_h)$. However, it is possible to 
obtain it if $f$ satisfies
\be\label{derivadamenosinfty}\lim_{m'\downarrow 0} 
f(x,m')=-\infty \hspace{0.8cm} \forall  \; x\in \TT^d,\ee
which is satisfied, for example if $F(x,m)=m\log m 
+ mF^1(x_{i,j})$ with $F^1$ continuous in $\TT^2$.

\begin{proposition}\label{derivadainfty} Suppose that $\nu=0$ 
and that \eqref{derivadamenosinfty} holds. Then, for every 
solution 
$(m^h,w^h)$ to  $(P_h^{d})$ or $(P_h)$, we have $m_{i,j}^h>0$ for all 
$i$,$j$. Consequently, the conclusion of Corollary 
\ref{casomasagennuigual0} holds and we have  the equality for 
the first 
equations in \eqref{firsorderinequality1} and 
\eqref{firsorderinequality2}.
\end{proposition}
\begin{proof} Since the argument is the same for both problems, 
we consider only problem $(P_h)$. Suppose the existence of 
$(i,j)$ such that $m^{h}_{i,j}=0$. Then, since the cost function 
is finite at $(m^h, w^h)$, we must have that  $w^{h}_{i,j}=0$ 
and, by feasibility, there exists $(i',j')$ such that 
$m_{i',j'}^h>0$. For any
$0<\delta<m_{i',j'}^h$,  define $\hat{m}$ by 
$\hat{m}_{i,j}=\delta$, 
$\hat{m}_{i',j'}=m^{h}_{i',j'}-\delta$ and $\hat{m}_{i'',j''}= 
m_{i'',j''}^{h}$ for all $(i'',j'')\notin \{(i,j), (i',j')\}$. 
Clearly, $(\hat{m}, w^h)$ is feasible for problem $(P_h)$ and 
the difference of the cost function for    $(\hat{m}, w^h)$ and 
$(m^h, w^h)$ is given by \small
\be\label{diferenciadecostos}
\hat{b}(m^{h}_{i',j'}-\delta, w^{h}_{i',j'})- \hat{b}(m^{h}_{i',j'}, w^{h}_{i',j'}) + F(x_{i',j'},m^{h}_{i',j'}-\delta)- F(x_{i',j'},m^{h}_{i',j'})+ F(x_{i,j},\delta)- F(x_{i,j},0).\ee \normalsize
From the Mean Value Theorem we have $F(x_{i,j},\delta)- 
F(x_{i,j},0)= f(x_{i,j}, \hat{\delta})\delta$, for some 
$\hat{\delta}\in (0, \delta)$, and since the first two 
differences in \eqref{diferenciadecostos} are of order 
$O(\delta)$, we get that the expression in 
\eqref{diferenciadecostos} is strictly negative if $\delta$ is 
small enough. This contradicts the optimality of $(m^h, w^h)$.  
Consequently, since $m^h>0$, we have $(\nabla 
\F^{+}(m^h))_{i,j}=f(x_{i,j},m_{i,j}^{j})\in \RR$ for all $i,j$, 
and we can reproduce the proof in Theorem  
\ref{condicionesoptimalidad} to establish 
\eqref{firsorderinequality2} with $\mu_{i,j}=0$ for all $i,j$.
\end{proof}
\subsection{The dual of the discrete problem} Throughout the rest of the paper we 
assume that 
$F(x_{i,j},\cdot)$ is 
convex for all $i$, $j$ (equivalently, 
$f(x_{i,j},\cdot)$ is increasing for all $i$, $j$). In this case, we derive the dual 
problem associated to $(P_h^{d})$ and 
$(P_h)$. 
Using the notation \eqref{definicionfuncompletas}, we must first 
calculate $(\B+ \F)^{*}$. Clearly, for $(\alpha, \beta) \in \M_h \times \W_h$ we have
$$(\B+ \F)^{*}(\alpha,\beta)= \sum_{i,j}(\hat{b}+ 
F(x_{i,j},\cdot))^{*}(\alpha_{i,j},\beta_{i,j}).  $$
By chosing $(\tilde{m}, \tilde{w})$ as in the proof of Theorem 
\ref{condicionesoptimalidad} and applying \cite[Theorem 
9.4.1]{attouch}, we have, for every $i,j$, 
$$(\hat{b}+ 
F(x_{i,j},\cdot))^{*}(\alpha_{i,j},\beta_{i,j})=\inf_{(\alpha',\beta')\in
\RR\times \RR^{4}}\left\{ \hat{b}^{\ast}(\alpha_{i,j}-\alpha', 
\beta_{i,j}-\beta')+ 
F^{\ast}(x_{i,j},\alpha',\beta')\right\}, $$
where $F(x_{i,j},\cdot)$ is seen as a function of $(m,w)$, constant 
in $w$. It is easy to check that 
$F^{*}(x_{i,j},a,b)=F^{*}(x_{i,j},a)$ if $b=0$ and  
$F^{*}(x_{i,j},a,b)=+\infty$ otherwise. Thus, by using Lemma 
\ref{l:partial} we obtain
$$\ba{rcl} (\hat{b}+ 
F(x_{i,j},\cdot))^{*}(\alpha_{i,j},\beta_{i,j})&=&\inf_{\alpha' 
\in 
\RR}\left\{F^{\ast}(x_{i,j},\alpha') \; ; \; \alpha_{i,j} + 
\frac{1}{q'}|P_{K}(\beta_{i,j})|^{q'} \leq 
\alpha'\right\},\\[4pt]
\; &=& F^{\ast}\left(x_{i,j},\alpha_{i,j} + 
\frac{1}{q'}|P_{K}(\beta_{i,j})|^{q'}\right), \ea$$
where in the last equality we have used that $F^{\ast}(x_{i,j}, 
\cdot)$ is increasing. Let us define $\Xi: \M_{h} \times \W_{h} \to \M_{h}\times \RR \times \M_{h}$ as  $\Xi(m,w)= \left(G(m,w), m\right)$.  Using that 
$\Xi^{*}(u,\lambda,p)=(\nu A^{*}u+h^{2}\lambda{\bf 
1}_{\M_h}+p, 
B^{*}u)$ 
we get that the Fenchel-Rockafellar dual problem 
\cite[Definition~15.19]{BausCombettes11} is 
given by
$$\ba{l} \inf
\left\{\lambda+\sup_{m\in \D} \sum_{i,j} p_{i,j}m_{i,j}+ 
\sum_{i,j} F^{\ast}\left(x_{i,j}, (-\nu A^{*}u)_{i,j} 
+\frac{1}{q'}|P_{K}(-(B^{*}u)_{i,j})|^{q'} -
p_{i,j}- \lambda 
h^2  \right)\right\}\\[6pt]
=\inf
\left\{\lambda+\sup_{m\in \D} \sum_{i,j} p_{i,j}m_{i,j}+ 
\sum_{i,j} F^{\ast}\left(x_{i,j}, (\nu A^{*}u)_{i,j} 
+\frac{1}{q'}|P_{K}((B^{*}u)_{i,j})|^{q'} -p_{i,j}- \lambda 
h^2  \right)\right\},
\ea$$ \normalsize
where the infimum  is taken over all $(u,\lambda,p)\in \M_h \times \RR \times \M_{h}$. Using that 
$$\sup_{m\in \D} \sum_{i,j} p_{i,j}m_{i,j}= \left\{\ba{ll} \sum_{i,j} p_{i,j}d_{i,j} & \mbox{if } p\geq 0, \\[4pt]
									     +\infty & \mbox{otherwise,}\ea\right.$$
we get that the dual problem is given by (compare with \cite[Proposition 4.5]{Mészáros2015} in the continuous framework)
\be\label{problemadualconrestricciondecaja}
\inf
\left\{\lambda+\sum_{i,j} p_{i,j}d_{i,j}+ \sum_{i,j} F^{\ast}\left(x_{i,j}, -\nu(\Delta_h 
u)_{i,j} +\frac{1}{q'}| \widehat{[D_h u ]}_{i,j}|^{q'} 
- p_{i,j} - \lambda 
h^2  \right)\right\}
\ee
where the infimum is taken over all $(u,\lambda,p)\in \M_h 
\times \RR \times \M_{h}$ satisfying that that $p\geq 0$.  It 
follows from Lemma~\ref{separacionsubdiferencial} and classical 
results in finite-dimensional convex duality theory (see e.g. 
\cite{MR0274683}) that the dual problem has at least one 
solution $(u,\lambda,p)$ and that the optimal value of $(P_h)$ 
equals minus the value in 
\eqref{problemadualconrestricciondecaja} (no duality gap).

If we do not consider box constraints (i.e. $d_{i,j}=+\infty$ for all $i$, $j$), analogous computations yield  that the dual problem is given by 
\be\label{problemadual} \min_{(u,\lambda)\in \M_h \times \RR} 
\left\{\lambda+\sum_{i,j} F^{\ast}\left(x_{i,j}, -\nu(\Delta_h 
u)_{i,j} +\frac{1}{q'}| \widehat{[D_h u ]}_{i,j}|^{q'}  - \lambda 
h^2  \right)\right\}\ee
and that this problem admits at least one solution $(u,\lambda)$.

In the convex case the results in Theorem  
\ref{condicionesoptimalidad} can be retrieved from this dual 
formulation using that the 
primal and dual problems admit solutions and that there is no 
duality gap (see \cite{Mészáros2015} for this type of argument 
in the continuous case and \cite{AchdouCapuzzoCamilli10} in the 
context of the discretization of the so-called planning problem 
in MFG).

\section{Iterative algorithms for solving $(P_h^{d})$ and 
$(P_h)$}\label{als}
In this section we review some proximal splitting methods for 
solving optimization problems and we provide their application to 
$(P_h^{d})$ and $(P_h)$. We also obtain a new splitting method which 
avoid matrix inversions. From 
now on, we assume that $f$ is increasing with respect to the 
second variable. Hence, the objective functions 
of these problems are convex and non-smooth, which lead us to 
focus in methods performing implicit instead of gradient steps.
 The performance of these splitting algorithms rely on the efficiency 
 on the computation of the implicit steps, in which the proximity 
 operator arises naturally. Let us recall that, for any convex 
 l.s.c. function
$\varphi\colon\RR^N\to\RX$ (eventually non-smooth), and $x\in\RR^N$, 
there exists a unique 
solution to 
\begin{equation}
\minimize{y\in\RR^N}{\varphi(y)+\frac{1}{2}|y-x|^2},
\end{equation}
which is denoted by $\prox_{\varphi}x$. The {\em proximity operator}, denoted by $\prox_{\varphi}$,
associates $\prox_{\varphi}x$ to each $x\in\RR^N$. From classical convex analysis we have
\begin{equation}
\label{e:prox}
p=\prox_{\varphi}x\quad\Leftrightarrow\quad 
x-p\in\partial\varphi(p),
\end{equation}
where $\partial\varphi$ stands for the subdifferential operator of 
$\varphi$ defined in \eqref{subdiferencialdef}.

\subsection{Proximal splitting algorithms}
\label{sec:5.meth}
For understanding the 
meaning of a proximal step, let $\gamma>0$, $x_0\in\RR^N$, suppose 
that $\varphi$ is differentiable, and
consider the {\em proximal point 
algorithm} \cite{Martinet70, Rock76}
\begin{equation}
\label{e:ppa}
(\forall n\geq 0)\quad x_{n+1}=\prox_{\gamma\varphi}x_n.
\end{equation}
 In this case, it follows from \eqref{e:prox} that 
\eqref{e:ppa} is equivalent to
\begin{equation}
\frac{x_{n}-x_{n+1}}{\gamma}=\nabla\varphi(x_{n+1}),
\end{equation}
which is an implicit discretization of the gradient flow. Then, the 
proximal iteration can be seen as an implicit step, which can be 
efficiently computed in several cases (see e.g., 
\cite{CombettesWajs}). 
The sequence generated by this algorithm (even in the non-smooth 
convex case) converges to a 
minimizer of $\varphi$  whenever it exists. 
However, since our problem involves constraints, it is natural that 
the methods for solving $(P_h)$ or $(P_h^{d})$ should 
involve them and, therefore, are more complicated.

Let us start with a general setting.
Let $\varphi\colon\RR^N\to\RX$ and $\psi\colon\RR^M\to\RX$ be two 
convex 
l.s.c. proper functions, and let 
$\Xi\colon\RR^N\to\RR^M$ a linear operator ($M\times N$ real 
matrix). Consider the optimization problem 
\begin{equation}
\label{e:primal}
\minimize{y\in\RR^N}{\varphi(y)+\psi(\Xi y)}
\end{equation}
and the associated Fenchel-Rockafellar dual problem
\begin{equation}
\label{e:dual}
\minimize{\sigma\in\RR^M}{\psi^*(\sigma)+\varphi^*(-\Xi^{*} 
\sigma)}.
\end{equation}
We have that 
\eqref{e:primal} and \eqref{e:dual} can be equivalently 
formulated as 
\begin{equation}
\label{e:primalADMM}
\min_{y\in\RR^N\!\!,\,v\in\RR^M}\varphi(y)+\psi(v)\quad 
\text{s.t.}\quad \Xi y=v
\end{equation}
and 
\begin{equation}
\label{e:dualADMM}
\min_{z\in\RR^N\!\!,\,\sigma\in\RR^M}\psi^*(\sigma)+\varphi^*(z)\quad
 \text{s.t.}\quad  -\Xi^{*}\sigma=z,
\end{equation}
respectively. Moreover, under qualification conditions (satisfied in our setting),
any primal-dual solution $(\hat{y},\hat{\sigma})$ to 
\eqref{e:primal}-\eqref{e:dual} satisfies, for every $\gamma>0$ and 
$\tau>0$,
\begin{equation}
\label{e:caract}
\begin{cases}
-\Xi^{*}\hat{\sigma}\in\partial\varphi(\hat{y})\\
\Xi\hat{y}\in\partial\psi^*(\hat{\sigma})
\end{cases}\Leftrightarrow\quad
\begin{cases}
\hat{y}-\tau\Xi^{*}\hat{\sigma}\in\tau\partial\varphi(\hat{y})+\hat{y}\\
\hat{\sigma}+\gamma\Xi\hat{y}\in\gamma\partial\psi^*(\hat{\sigma})+\hat{\sigma}
\end{cases}\Leftrightarrow\quad
\begin{cases}
\prox_{\tau\varphi}(\hat{y}-\tau\Xi^{*}\hat{\sigma})=\hat{y}\\
\prox_{\gamma\psi^*}(\hat{\sigma}+\gamma\Xi\hat{y})=\hat{\sigma}.
\end{cases}
\end{equation}

In the particular case when
\be\label{def.phi}
\varphi\colon(m,w)\mapsto\sum_{i,j=1}^{N_h}\hat{b}(m_{i,j},w_{i,j})
+F(x_{i,j},m_{i,j})+\iota_{\mathcal{D}}(m_{i,j}),\ee
 $(P_h^d)$ can be recast  as 
\eqref{e:primal} via two formulations.
\begin{itemize}
\item[\textbullet] {\it Without splitting.} We consider $N=M=5\times(N_h\times N_h)$,  $\Xi=\Id$, and
\be\label{e:defV}
\psi=\iota_{V}, \; \; \mbox{with } \; V=\menge{(m,w)\in \RR^{N}}{G(m,w)=(0,1)}= ({\bf 
1}_{\M_h},0)+\mbox{ker } G,\ee
where we recall that $G$ is defined in \eqref{defGyD}.
\item[\textbullet] {\it With splitting.}  We split the 
influence of linear operators from $\psi$ by considering  $N=5 
N_h^2$, $M=N_h^2+1$, $\psi=\iota_{\{(0,1)\}}$  and $\Xi=G$.
\end{itemize}
In the latter case, the dual problem \eqref{e:dual} reduces to  \eqref{problemadual}.  In the next section, we will see that the two formulations lead to different algorithms.  
In the rest of this section we recall some classical algorithms 
to solve \eqref{e:primal}.   For the sake simplicity, we specify 
the computation of the steps of each algorithm under the 
formulation {\it without splitting} for problem $(P_h^d)$. 
\subsubsection{Alternating direction method of multipliers (ADMM)}\label{admm}
In this part we briefly recall the ADMM 
\cite{MR0388811,GABAY197617,Gabay1983299}, which is a variation of 
the Augmented Lagrangian Algorithm introduced in 
\cite{MR0262903,MR0271809,MR0272403} (see references 
\cite{Boyd:2011:DOS:2185815.2185816,Ecksteinphd} for two surveys on 
the subject). The algorithm can be seen as an application of Douglas-Rachford splitting to \eqref{e:dual} \cite{Gabay1983299,EcksteinBertsekas92}. Problem \eqref{e:primalADMM} can be written 
equivalently as 
\begin{equation}
\label{e:primalADMM2}
\min_{y\in\RR^N\!\!,\,v\in\RR^M}\max_{\sigma\in\RR^M}L(y,v,\sigma),
\end{equation}
where $L: 
\RR^{N}\times \RR^{M} \times \RR^{M} \mapsto ]-\infty, +\infty]$ 
is the Lagrangian associated to \eqref{e:primalADMM},  
defined by
\begin{equation}
\label{e:defLagrangian}
L(y,v,\sigma)= \varphi(y) + \psi(v) + \sigma \cdot (\Xi 
y-v).
\end{equation}
Given $\gamma>0$, the {\it augmented Lagrangian} 
$L_{\gamma}: 
\RR^{N}\times \RR^{M} \times \RR^{M} \mapsto ]-\infty, +\infty]$ is 
defined by
\begin{equation}
L_{\gamma}(y,v,\sigma)= L(y,v,\sigma)+ \frac{\gamma}{2}|\Xi 
y-v|^{2}.
\end{equation}
Given an initial point $(y^{0},v^{0},\sigma^{0})$, the iterates of 
ADMM are obtained by the following procedure: for every 
$k\geq0$, 
\be\label{pasoslagrangianoaumentado}
\ba{rcl} y^{k+1}&=& 
\mbox{argmin}_{y}L_{\gamma}(y,v^k,\sigma^k)=\mbox{argmin}_{y}\left\{ 
\varphi(y)+ 
\sigma^{k}\cdot \Xi 
y + \frac{\gamma}{2}|\Xi y-v^{k}|^{2}\right\}\\[6pt]
v^{k+1}&=& 
\mbox{argmin}_{v}L_{\gamma}(y^{k+1},v,\sigma^k)=\prox_{\psi/\gamma}(\sigma^k/\gamma+\Xi
y^{k+1})\\[6pt]
\sigma^{k+1}&=&\mbox{argmax}_{\sigma}\left\{L(y^{k+1},v^{k+1},\sigma)
-\frac{1}{2\gamma}|\sigma-\sigma^k|^2\right\}=
 \sigma^{k}+ \gamma(\Xi y^{k+1}-v^{k+1}).
\ea\ee
This algorithm is simple to implement in the case when $\varphi$ 
is a 
quadratic function, in which case the first step in 
\eqref{pasoslagrangianoaumentado} reduces 
to solve a linear system. This is the case in several problems in PDE's,
where this method is widely used. However, for general convex functions 
$\varphi$, the first step in \eqref{pasoslagrangianoaumentado} is not 
always easy to compute. 
Indeed, it has not closed expression for most of combinations of 
convex functions $\varphi$ and matrices $\Xi$, even if 
$\prox_{\varphi}$ is computable, which leads to subiterations in 
those cases. Moreover, it needs a full {\color{red}column}-rank assumption on $\Xi$ for 
achieving convergence. However, in 
some particular cases, it can be solved 
efficiently. For instance, assume that 
\begin{equation}
(\forall i,j\in\{1,\ldots, N_h\})\quad 	F(x_{i,j},m)=\begin{cases}
	g_{i,j}(m),\quad&\text{if}\quad m\geq 0\\
	+\infty,&\text{otherwise},
	\end{cases}
\end{equation}
where $g_{i,j}\colon\RR\to\RR$ is a differentiable strictly 
convex
function satisfying $\lim_{m\to+\infty}g_{i,j}'(m)=+\infty$.
By recalling that 
$\widehat{[D_h u 
]}_{i,j}=P_{K}(-[D_hu]_{i,j})$ and that 
$(B^{*}u)_{i,j}=-[D_hu]_{i,j}$, the dual problem 
\eqref{problemadual} (for $q=2$) 
reduces to \eqref{e:primal} by choosing $N=N_h^2+1$,
$M=6N_h^2$, 
${\bf 
1}_{\M_h}\in\mathcal{M}_h$,
 

\begin{equation}\label{definicionfuncionesADMM}
\begin{cases}
y=(u,\lambda) \in \RR^N, &\quad  \varphi(y)= \lambda,\quad 
\Xi 
y=(-h^{2}\lambda{\bf 
1}_{\M_h},B^{*}u,\nu A^{*}u)\\
v=(a,b,c) \in \RR^M,&\quad \psi(v)=\sum_{i,j}\phi_{i,j}( 
a_{i,j}+\frac{1}{2}|P_{K}b_{i,j}|^{2}+c_{i,j}),
\end{cases}
\end{equation}
where, for every $i,j\in\{1,\ldots, N_h\}$,
\begin{equation}
	\label{e:defphi}
	\phi_{i,j}(\eta):=(F(x_{i,j},\cdot))^*(\eta)
	=\begin{cases}
	g_{i,j}^*(\eta),\quad&\text{if }\eta > g_{i,j}'(0);\\
	-g_{i,j}(0),&\text{otherwise}.
	\end{cases}
\end{equation}
Note that the assumptions on $g_{i,j}$ imply that $g_{i,j}^*$ is 
a strictly convex differentiable function on ${\rm int}\dom 
g_{i,j}^*\supset]g_{i,j}'(0),+\infty[$ (see, e.g., 
\cite[Theorem~26.3]{MR0274683}), 
 and, hence, $\phi_{i,j}$ is non-decreasing, convex, and 
 differentiable. 
Therefore, by using first order optimality conditions, the 
steps in 
\eqref{pasoslagrangianoaumentado} reduce to
\begin{align}
y^{k+1}&=
\begin{pmatrix}
u^{k+1}\\
\lambda^{k+1}
\end{pmatrix}=
\begin{pmatrix}
(\nu^2AA^{*}+BB^{*})^{-1}\Big(B(b^k-\sigma^k_2/\gamma)+\nu 
A(c^k-\sigma^k_3/\gamma)\Big)\\
\frac{1}{\gamma }\sum_{i,j}\big(\sigma_{1,i,j}^k-1-\gamma 
a_{i,j}^k\big)
\end{pmatrix}\\
v^{k+1}&=\begin{pmatrix}
a^{k+1}\\
b^{k+1}\\
c^{k+1}
\end{pmatrix}=\prox_{\psi/\gamma}
\begin{pmatrix}
\sigma_1^{k}/\gamma-h^2\lambda^{k+1}{\bf 
1}_{\M_h}\\
\sigma_2^{k}/\gamma+B^{*}u^{k+1}\\
\sigma_3^{k}/\gamma+\nu A^{*}u^{k+1}
\end{pmatrix}
\label{e:proxadmm}\\
\sigma^{k+1}&=\begin{pmatrix}
\sigma_1^{k+1}\\
\sigma_2^{k+1}\\
\sigma_3^{k+1}
\end{pmatrix}=
\begin{pmatrix}
\sigma_1^{k}-\gamma (h^2\lambda^{k+1}{\bf 
1}_{\M_h}+a^{k+1})\\
\sigma_2^{k}+\gamma(B^{*}u^{k+1}-b^{k+1})\\
\sigma_3^{k}+\gamma(\nu A^{*}u^{k+1}-c^{k+1})
\end{pmatrix}.
\end{align}
The more difficult 
step for 
ADMM 
is \eqref{e:proxadmm}, whose explicit 
calculation is the next result. 
\begin{lemma}
Let $\gamma>0$. We have $\prox_{\psi/\gamma}\colon (a,b,c)\mapsto 
(\prox_{\psi_{i,j}/\gamma}(a_{i,j},b_{i,j},c_{i,j}))_{i,j}$, where
\begin{equation}
\label{e:proxADMMpsi}
\prox_{\psi_{i,j}/\gamma}(\alpha_0,\beta_0,\delta_0)=
\begin{cases}
\begin{pmatrix}
\alpha_0-s_{i,j}\\
\frac{P_K\beta_0}{1+s_{i,j}}+P_{K^-}\beta_0\\
\delta_0-s_{i,j}
\end{pmatrix},
\quad&\text{{\rm if} }\alpha_0+|P_K\beta_0|^2/2+\delta_0> 
g_{i,j}'(0);\\
(\alpha_0,\beta_0,\delta_0),&\text{{\rm otherwise}},
\end{cases}
\end{equation}
where $s_{i,j}$ is the unique non-negative solution to the 
equation on $s$:
\begin{equation}
\label{e:NLequation}
\gamma s=(g_{i,j}^*)'\Big(\alpha_0+\delta_0-2s
+\frac{|P_K\beta_0|^2}{2(1+s)^2}\Big).
\end{equation}
\end{lemma}
\begin{proof} We adapt the argument in 
\cite[Appendix]{benamoucarlier15} for considering the more 
general functions $g_{i,j}$ and the  
presence of the set $K$ in the definition of $\psi$ in 
\eqref{definicionfuncionesADMM}. 
Since 
$\psi=\sum_{i,j}\psi_{i,j}$ is separable, we have from 
\cite[Proposition~23.30]{BausCom11}
that $\prox_{\psi/\gamma}\colon 
(a,b,c)\mapsto(\prox_{\psi_{i,j}/\gamma}(a_{i,j},b_{i,j},c_{i,j}))_{i,j}$,
 where
$\psi_{i,j}(\alpha,\beta,\delta):=\phi_{i,j}(\alpha+\frac{1}{2}|P_{K}\beta|^2+\delta)$.
Using that 
$\nabla(|P_K\beta|^2/2)=P_K\beta$ for every $\beta\in\RR^4$ and 
\eqref{e:prox}, we have 
\begin{equation}
\label{primeraeqprocla}
\begin{pmatrix}
\alpha\\
\beta\\
\delta
\end{pmatrix}
=\mbox{prox}_{\psi_{i,j}/\gamma}
\begin{pmatrix}
\alpha_0\\
\beta_0\\
\delta_0
\end{pmatrix}
\quad 
\Leftrightarrow\quad 
\begin{pmatrix}
\alpha_0-\alpha\\
\beta_0-\beta\\
\delta_0-\delta
\end{pmatrix}= 
\frac{1}{\gamma}\nabla\psi_{i,j}(\alpha,\beta,\delta)=
\frac{1}{\gamma}\phi_{i,j}'(\alpha+|P_{K}\beta|^{2}/2+\delta)
\begin{pmatrix}
1\\
P_K\beta\\
1
\end{pmatrix}.
\end{equation}
By denoting 
$s_{i,j}=\phi_{i,j}'(\alpha+|P_{K}\beta|^{2}/2 
+\delta)/\gamma\geq0$,
we deduce from \eqref{primeraeqprocla} that 
$\alpha=\alpha_0-s_{i,j}$, $\delta=\delta_0-s_{i,j}$ and, for every 
$\ell\in\{1,\ldots,4\}$, 
$\beta_{\ell}=\frac{P_{K_{\ell}}(\beta_{0})_{\ell}}{1+s_{i,j}}+
P_{-K_{\ell}}(\beta_{0})_{\ell}$, where $K_1=K_3=\RR_+$ and 
$K_2=K_4=\RR_-$.
In other words,
$\beta=\frac{P_{K}\beta_{0}}{1+s_{i,j}}+P_{K^{-}}\beta_{0}$. 
On the other hand,
if $\alpha+|P_{K}\beta|^{2}/2+\delta > g_{i,j}'(0)$, it 
follows from  \eqref{e:defphi} and the previous relations that
\begin{align}
\gamma 
s_{i,j}&=\phi_{i,j}'(\alpha+|P_{K}\beta|^{2}/2+\delta)\nonumber\\
&=(g_{i,j}^*)'\big(\alpha+|P_{K}\beta|^{2}/2+\delta\big)\nonumber\\
&=(g_{i,j}^*)'\Big(\alpha_0+\delta_0-2s_{i,j}+\frac{|P_{K}\beta_0|^{2}}{2(1+s_{i,j})^2}\Big).
\end{align} 
Otherwise, if $\alpha+|P_{K}\beta|^{2}/2+\delta\leq g_{i,j}'(0)$,
$\gamma s_{i,j}=\phi_{i,j}'(\alpha+|P_{K}\beta|^{2}/2+\delta)
=0$, which yields $\alpha=\alpha_0$, $\beta=\beta_0$, and 
$\delta=\delta_0$. Hence, 
$\alpha_0+|P_{K}\beta_0|^{2}/2+\delta_0\leq g_{i,j}'(0)$ and
the result follows. 
\end{proof}
\begin{remark}
Note that, by defining, for every $i,j$, $h_{i,j}\colon 
s\mapsto\gamma s-(g_{i,j}^*)'(\alpha_0+\delta_0-2s
+\frac{|P_{K}\beta_0|^{2}}{2(1+s)^2})$, we have 
$h_{i,j}(0)=-(g_{i,j}^*)'(\alpha_0+\delta_0
+|P_{K}\beta_0|^{2}/2)<-(g_{i,j}^*)'(g_{i,j}'(0))=0$, since
$\alpha_0+\delta_0
+|P_{K}\beta_0|^{2}/2>g_{i,j}'(0)$ and 
$(g_{i,j}^*)'=(g_{i,j}')^{-1}$ is 
strictly increasing. Moreover, it is easy to check that 
$h_{i,j}$ is strictly increasing and $h_{i,j}(\bar{s})>0$, where
$\bar{s}>0$ is the unique solution to $\alpha_0+\delta_0-2s
+\frac{|P_{K}\beta_0|^{2}}{2(1+s)^2}=g_{i,j}'(0)$. Hence, we 
deduce that \eqref{e:NLequation} has a unique solution in 
$]0,\bar{s}[$. Anyway, the existence and unicity of this 
equation can also be deduced from the unicity of $\prox_{\psi}$, 
since $\psi$ is 
proper, convex, and l.s.c.
\end{remark}

\begin{remark}
In particular, consider
	$q=q'=2$  and, for every $i,j\in\{1,\ldots,N_h\}$, 
	$g_{i,j}(m)=r(m-\bar{m}(x_{i,j}))^2/2$,
	where $\bar{m}: \TT^{2}\mapsto \RR$ is a given desired density 
	function and $r>0$ is a given constant. In this case
	\begin{equation}
		\phi_{i,j}(\eta)=(F(x_{i,j},\cdot))^*(\eta)= \left\{\ba{ll} 
		\frac{\eta^{2}}{2r}+ \eta 
		\bar{m}(x_{i,j}) & \mbox{{\rm if} } \eta \geq -r \bar{m}(x_{i,j}),\\[4pt]
		-r\frac{\bar{m}(x_{i,j})^{2}}{2}& \mbox{ {\rm 
		otherwise}},\ea\right.
	\end{equation}
condition in \eqref{e:proxADMMpsi} changes to 
$\alpha_0+|P_{K}\beta_0|^{2}/2+\delta_0\geq-r\bar{m}(x_{i,j})$, 
and 
\eqref{e:NLequation}
reduces to 
\begin{equation}
\gamma s_{i,j}
=\bar{m}(x_{i,j})+\frac{1}{r}\Big(\alpha_0+\delta_0-2s_{i,j}+
\frac{|P_{K}\beta_0|^{2}}{2(1+s_{i,j})^2}\Big).
\end{equation} 
	The ADMM with this type of quadratic functions have been used to solve optimal 
	transport 
	problems in {\rm\cite{MR1738163,MR2516195}} and recently in 
	{\rm\cite{benamoucarlier15}}  in the context of static and time-dependent 
	mean field games. Since diffusion terms and the set $K$ are not considered in 
	{\rm\cite{benamoucarlier15}}, the computation of the 
	proximity operator in our case 
	differs from {\rm\cite[Section~7]{benamoucarlier15}}.
	\end{remark}

\subsubsection{Predictor-corrector proximal multiplier method 
(PCPM)}\label{pcpmu}
Another approach for solving \eqref{e:primalADMM} is proposed by 
Chen 
and Teboulle in \cite{ChenTeb94}. Given $\gamma>0$ and starting 
points 
$(y_0,v_0,\sigma_0)\in\RR^N\times\RR^M\times\RR^M$ iterate
\be\label{pasosChenTeb}
\ba{rcl} p^{k+1}&=& 
\mbox{argmax}_{\sigma}\left\{L(y^{k},v^{k},\sigma)
-\frac{1}{2\gamma}|\sigma-\sigma^k|^2\right\}=
 \sigma^{k}+ \gamma(\Xi y^{k}-v^{k}),\\
y^{k+1}&=&\mbox{argmin}_{y}\left\{L(y,v^k,p^{k+1})
+\frac{1}{2\gamma}|y-y^k|^2\right\}=
\prox_{\gamma\varphi}(y^{k}-\gamma \Xi^{*}p^{k+1}),\\
v^{k+1}&=& 
\mbox{argmin}_{v}\left\{L(y^{k},v,p^{k+1})+\frac{1}{2\gamma}|v-v^k|^2\right\}=\prox_{\gamma\psi}(v^{k}+\gamma
 p^{k+1}),\\[6pt]
\sigma^{k+1}&=&\mbox{argmax}_{\sigma}\left\{L(y^{k+1},v^{k+1},\sigma)
-\frac{1}{2\gamma}|\sigma-\sigma^k|^2\right\}=
 \sigma^{k}+ \gamma(\Xi y^{k+1}-v^{k+1}),
\ea\ee
where the Lagrangian $L$ is defined in \eqref{e:defLagrangian}.
In comparison to ADMM, after a prediction of the multiplier in 
the first step, this method performs an 
additional correction step  on the 
dual variables and parallel updates on 
the primal ones by using the standard Lagrangian with an
additive inertial quadratic term instead of the augmented 
Lagrangian. This feature allows us to perform only 
explicit steps, if $\prox_{\gamma\varphi}$ and 
$\prox_{\gamma\psi}$ can be computed easily, overcoming one of 
the problems of ADMM. The convergence to a solution to \eqref{e:primalADMM}
is obtained provided that $\gamma\in\left]0,\min\{1,\|\Xi\|^{-1}\}/2\right[$.

In the formulation without splitting, 
we have from \eqref{e:defV} that
\be\label{e:defPV}
\prox_{\gamma\psi}=P_V=\Id-G^*(GG^*)^{-1}G(\Id-({\bf 1}_{\mathcal{M}_h},0))\ee
and, since $A{\bf 1}_{\mathcal{M}_h}=0$, we obtain 
$$GG^*=\begin{pmatrix}
\nu^2AA^{*}+BB^{*}&0\\
0& h^2
\end{pmatrix}.$$ 
Hence, by denoting $y=(m,w)$, $\sigma=(n,x)$, and 
${v}=(\bar{n},\bar{x})$,
\eqref{pasosChenTeb} reduces to 
\begin{align}
\begin{pmatrix}
p_1^{k+1}\\
p_2^{k+1}
\end{pmatrix}
&=
\begin{pmatrix}
n^{k}+\gamma(m^k-\bar{n}^k)\\
x^{k}+\gamma(w^k-\bar{x}^k)
\end{pmatrix}\nonumber\\
\begin{pmatrix}
m^{k+1}\\
w^{k+1}
\end{pmatrix}
&=
\prox_{\gamma\varphi}
\begin{pmatrix}
m^{k}-\gamma p_1^{k+1}\\
w^{k}-\gamma p_2^{k+1}
\end{pmatrix}\nonumber\\
\begin{pmatrix}
y^{k+1}\\
z^{k+1}
\end{pmatrix}
&=
\begin{pmatrix}
\bar{n}^k+\gamma p_1^{k+1}\\
\bar{x}^k+\gamma 
p_2^{k+1}
\end{pmatrix}\nonumber\\
\begin{pmatrix}
\bar{n}^{k+1}\\
\bar{x}^{k+1}
\end{pmatrix}
&=
\begin{pmatrix}
y^{k+1}-\nu 
A^{*}(\nu^2AA^{*}+BB^{*})^{-1}\big(\nu 
Ay^{k+1}+Bz^{k+1}\big)-(h^2\sum_{i,j}y_{i,j}^{k+1}-1){\bf 
1}_{\M_h}\\
z^{k+1}-B^{*}(\nu^2AA^{*}+BB^{*})^{-1}\big(\nu 
Ay^{k+1}+Bz^{k+1}\big)
\end{pmatrix}\nonumber\\
\begin{pmatrix}
{n}^{k+1}\\
{x}^{k+1}
\end{pmatrix}
&=
\begin{pmatrix}
{n}^k+\gamma (m^{k+1}-\bar{n}^{k+1})\\
{x}^k+\gamma (w^{k+1}-\bar{x}^{k+1})
\end{pmatrix},
\end{align}
where $\prox_{\gamma\varphi}$ will be computed in 
Proposition~\ref{p:calprox}.

\subsubsection{Chambolle-Pock's splitting (CP)}\label{cpu}

Inspired on the caracterization \eqref{e:caract} obtained from 
the 
optimality conditions, Chambolle and Pock in 
\cite{CPock11} propose an alternative primal-dual method for 
solving \eqref{e:primal} and \eqref{e:dual}. More precisely,
given $\theta\in[0,1]$, $\gamma>0$, $\tau>0$ 
and starting points
$(y^0,\bar{y}^0,\sigma^0)\in\RR^{N}\times\RR^{N}\times\RR^{M}$, 
the iteration 
\be\label{pasosChamPock}
\ba{rcl} \sigma^{k+1}&=& 
\mbox{argmax}_{\sigma}\left\{\mathcal{L}(\bar{y}^{k},\sigma)
-\frac{1}{2\gamma}|\sigma-\sigma^k|^2\right\}=
 \prox_{\gamma\psi^*}(\sigma^{k}+ \gamma\Xi\bar{y}^{k})\\
y^{k+1}&=&\mbox{argmin}_{y}\left\{\mathcal{L}(y,\sigma^{k+1})
+\frac{1}{2\tau}|y-y^k|^2\right\}=
\prox_{\tau\varphi}(y^{k}-\tau \Xi^{*}\sigma^{k+1})\\
\bar{y}^{k+1}&=& 
 y^{k+1}+ \theta(y^{k+1}-y^{k}),
\ea\ee
where 
$\mathcal{L}(y,\sigma):=\min_{v\in\RR^M}L(y,v,\sigma)=\sigma^{*}\Xi
y+\varphi(y)-\psi^*(\sigma)$,
generates a sequence $(y^k,\sigma^k)_{k\in\NN}$ which converges 
to a primal-dual solution to \eqref{e:primal}-\eqref{e:dual}.
Note 
that, if the proximity operators 
associated to $\varphi$ and $\psi^*$ are explicit, the method 
has 
only 
explicit steps, overcoming the difficulties of ADMM.  
For any $\theta\in]0,1]$, the last step of the 
method includes information of the last two 
iterations. 
The procedure of including memory on the algorithms has been 
shown to 
accelerate the methods for specific choice of the stepsizes (see 
\cite{Nest1,Nest2,BeckT}). The convergence of the 
method is obtained provided that $\tau\gamma<1/\|\Xi\|^2$.
 
Since in the formulation without splitting
 $\prox_{\gamma\psi*}=\Id-\gamma P_V\circ(\Id/\gamma)$ 
 \cite[Theorem~14.3(ii)]{BausCombettes11},
by using \eqref{e:defPV} and denoting $y=(m,w)$, $\sigma=(n,v)$, and 
$\bar{y}=(\bar{m},\bar{w})$, \eqref{pasosChamPock} reduces to
\begin{align}
\begin{pmatrix}
y^{k+1}\\
z^{k+1}
\end{pmatrix}
&=
\begin{pmatrix}
n^{k}+\gamma(\bar{m}^{k}-{\bf 
1}_{\M_h})\\
v^{k}+\gamma \bar{w}^{k}
\end{pmatrix}\nonumber\\
\begin{pmatrix}
n^{k+1}\\
v^{k+1}
\end{pmatrix}
&=
\begin{pmatrix}
\nu 
A^{*}(\nu^2AA^{*}+BB^{*})^{-1}\big(\nu 
Ay^{k+1}+Bz^{k+1}\big)+h^2\sum_{i,j}y_{i,j}^{k+1}{\bf 
	1}_{\M_h}\\
B^{*}(\nu^2AA^{*}+BB^{*})^{-1}\big(\nu 
Ay^{k+1}+Bz^{k+1}\big)
\end{pmatrix}\nonumber\\
\begin{pmatrix}
m^{k+1}\\
w^{k+1}
\end{pmatrix}
&=
\prox_{\tau\varphi}
\begin{pmatrix}
m^{k}-\tau n^{k+1}\\
w^{k}-\tau v^{k+1}
\end{pmatrix}\nonumber\\
\begin{pmatrix}
\bar{m}^{k+1}\\
\bar{w}^{k+1}
\end{pmatrix}
&=
\begin{pmatrix}
{m}^{k+1}+\theta({m}^{k+1}-{m}^{k})\\
{w}^{k+1}+\theta({w}^{k+1}-{w}^{k}),
\end{pmatrix}.
\end{align}
where, as before, $\prox_{\tau\varphi}$ will be computed in 
Proposition~\ref{p:calprox}.

\subsubsection{Monotone $+$ skew splitting method (MS)}\label{msu}
Alternatively, in \cite{Siopt1} a monotone 
operator-based approach is 
used, also inspired in the optimality conditions 
\eqref{e:caract}.  By calling $\mathcal{A}\colon 
(y,\sigma)\mapsto 
\partial\varphi(y)\times\partial\psi^*(\sigma)$ and 
$\mathcal{B}\colon (y,\sigma)\mapsto (\Xi^{*}\sigma,-\Xi y)$, 
\eqref{e:caract}  is equivalent to 
$(0,0)\in 
\mathcal{A}(\hat{y},\hat{\sigma})+\mathcal{B}(\hat{y},\hat{\sigma})$,
 where $\mathcal{A}$ is maximally monotone and $\mathcal{B}$ is 
skew linear. Under these conditions, Tseng in \cite{Tseng00} 
proposed an splitting algorithm for solving this monotone 
inclusion which performs two explicit steps on $\mathcal{B}$ and 
an implicit step on $\mathcal{A}$.
In our convex optimization context, given $\gamma>0$ 
and 
starting points $(y^0,\sigma^0)\in\RR^N\times\RR^M$,
the algorithm iterates, for every $k\geq 0$,
\be\label{pasosSiopt}
\ba{rcl} 
\eta^{k}&=& 
\prox_{\gamma\psi^*}(\sigma^{k}+ 
\gamma\Xi{y}^{k})=\mbox{argmax}_{\sigma}\left\{\mathcal{L}({y}^{k},\sigma)
-\frac{1}{2\gamma}|\sigma-\sigma^k|^2\right\}\\
p^{k}&=&
\prox_{\gamma\varphi}(y^{k}-\gamma 
\Xi^{*}\sigma^{k})=\mbox{argmin}_{y}\left\{\mathcal{L}(y,\sigma^{k})
+\frac{1}{2\gamma}|y-y^k|^2\right\}\\
\sigma^{k+1}&=&\eta^k+\gamma \Xi(p^k-y^{k})\\
{y}^{k+1}&=&p^k-\gamma \Xi^{*}(\eta^k-\sigma^{k}).
\ea\ee
Note that the updates on variables 
$\eta^k$ and $p^k$ can be performed in parallel. The convergence of the method is guaranteed if $\gamma\in\left]0,\|\Xi\|^{-1}\right[$.

Considering the formulation without splitting and proceeding 
analogously as in previous methods, 
\eqref{pasosSiopt} 
reduces to
\begin{align}
\begin{pmatrix}
y^{k+1}\\
z^{k+1}
\end{pmatrix}
&=
\begin{pmatrix}
n^{k}+\gamma({m}^{k}-{\bf 
1}_{\M_h})\\
v^{k}+\gamma {w}^{k}
\end{pmatrix}\nonumber\\
\begin{pmatrix}
\eta_1^{k}\\
\eta_2^{k}
\end{pmatrix}
&=
\begin{pmatrix}
\nu 
A^{*}(\nu^2AA^{*}+BB^{*})^{-1}\big(\nu 
Ay^{k+1}+Bz^{k+1}\big)+h^2\sum_{i,j}y_{i,j}^{k+1}{\bf 
	1}_{\M_h}\\
B^{*}(\nu^2AA^{*}+BB^{*})^{-1}\big(\nu 
Ay^{k+1}+Bz^{k+1}\big)
\end{pmatrix}\nonumber\\
\begin{pmatrix}
p_1^{k}\\
p_2^{k}
\end{pmatrix}
&=
\prox_{\gamma\varphi}
\begin{pmatrix}
m^{k}-\gamma n^k\\
w^{k}-\gamma v^{k}
\end{pmatrix}\nonumber\\
\begin{pmatrix}
{n}^{k+1}\\
{v}^{k+1}
\end{pmatrix}
&=
\begin{pmatrix}
{\eta}_1^{k}+\gamma(p_1^k-{m}^{k})\\
{\eta}_2^{k}+\gamma(p_2^k-{w}^{k})
\end{pmatrix}\nonumber\\
\begin{pmatrix}
{m}^{k+1}\\
{w}^{k+1}
\end{pmatrix}
&=
\begin{pmatrix}
{p}_1^{k}-\gamma(\eta_1^k-n^k)\\
{p}_2^{k}-\gamma(\eta_2^k-v^k)
\end{pmatrix}.
\end{align}

Note that in all previous algorithms, the inversion of the 
matrix $(\nu^2AA^{*}+BB^{*})$ is needed, which is usually badly 
conditioned in this 
type of applications depending on the viscosity paremeter $\nu$ 
(see the 
discussions in \cite{MR2928376,benamoucarlier15}). 
The inverse of $(\nu^2AA^{*}+BB^{*})$ is not needed in any of 
the previous 
methods if we use the formulation {\em with splitting}, i.e., if we split the 
influence of linear operators from $\psi$. However, in this case 
we obtain very slow algorithms, whose primal iterates usually do 
not 
satisfy any of the constraints. This motivates the 
following method 
which, by enforcing the iterates to satisfy (some of) the 
constraints via an additional projection step, has a better
performance than methods with splitting without any matrix 
inversion.

\subsubsection{Projected Chambolle-Pock splitting}\label{cpsp}
In this section we propose a modification of Chambolle-Pock 
splitting, whose convergence to a solution to $(P_h^d)$ is 
proved in the Appendix. This 
modification includes an additional projection step onto a set 
in 
which the solution has to be. In the case in which this set is an
affine 
vectorial subspace generated by (some of) the linear 
constraints, this 
modification allows us to guarantee that the generated iterates 
satisfy these constraints.

In order to present our algorithm in a general setting, let 
$C$ be closed convex subset of $\RR^N$ and consider 
the problem of finding a point in 
\begin{equation}
\Z=\menge{(y,\sigma)\in 
C\times\RR^M}{-\Xi^{*}{\sigma}\in\partial\varphi({y}),\:
\Xi{y}\in\partial\psi^*({\sigma})
}
\end{equation}
assuming $\Z\neq\varnothing$. Note that, from \eqref{e:caract}, 
every 
point in $\Z$ is a primal-dual solution to 
\eqref{e:primal}-\eqref{e:dual}. The 
following theorem provides the modified 
method and its convergence, whose proof can be 
found in the Appendix.
\begin{theorem}
\label{t:CPockproj}
Let $\gamma>0$ and $\tau>0$ be such that $\gamma\tau\|\Xi\|^2<1$ and 
let $(y^0,\bar{y}^0,\sigma^0)\in\RR^N\times\RR^N\times\RR^M$ be 
arbitrary starting points. For every $k\geq 0$ consider the routine
\begin{align}
\label{e:cpockproj}
\sigma^{k+1}&=\prox_{\gamma\psi^*}(\sigma^k+\gamma\Xi\bar{y}^k)\nonumber\\
p^{k+1}&=\prox_{\tau\varphi}(y^k-\tau\Xi^{*}\sigma^{k+1})\nonumber\\
y^{k+1}&=P_C\,p^{k+1}\nonumber\\
\bar{y}^{k+1}&= y^{k+1}+ \theta(p^{k+1}-y^{k}).
\end{align}
Then, there exists $(\hat{y},\hat{\sigma})\in\Z$  such that 
$y^k\to\hat{y}$ 
and
$\sigma^k\to\hat{\sigma}$. 
\end{theorem}

In order to focus only on the
projection onto the constraint 
$h^2\sum_{i,j=1}^{N_h}m_{i,j}=1$, we consider the  
formulation {\em with splitting} detailed in Section~\ref{sec:5.meth} 
and the previous method 
with    
$$C=\Menge{(m,w)\in\RR^{N}}{h^2\sum_{i,j=1}^{N_h}m_{i,j}=1}.$$
We 
obtain $\prox_{\gamma\psi^*}=\Id-\gamma(0,1)$ and 
$P_C\colon(m,w)\mapsto ({\bf 
1}_{\M_h}+(m-h^2\sum_{i,j=1}^{N_h}m_{i,j}{\bf 
1}_{\M_h}),w)$
and, hence, by denoting $\sigma=(u,\lambda)$, $y=(m,w)$,
$\bar{y}=(\bar{m},\bar{w})$, $p=(n,v)$, \eqref{e:cpockproj} 
reduces to
\begin{align}
\begin{pmatrix}
u^{k+1}\\
\lambda^{k+1}
\end{pmatrix}
&=
\begin{pmatrix}
u^{k}+\gamma(A\bar{m}^{k}+B\bar{w}^{k})\\
\lambda^{k}+\gamma (h^2\sum_{i,j}\bar{m}_{i,j}^{k}-1)
\end{pmatrix}\nonumber\\
\begin{pmatrix}
n^{k+1}\\
v^{k+1}
\end{pmatrix}
&=
\prox_{\tau\varphi}
\begin{pmatrix}
m^{k}-\tau(A^{*}u^{k+1}+h^2\lambda^{k+1}{\bf 
1}_{\M_h})\\
w^{k}-\tau B^{*}u^{k+1}
\end{pmatrix}\nonumber\\
\begin{pmatrix}
{m}^{k+1}\\
{w}^{k+1}
\end{pmatrix}
&=
\begin{pmatrix}
{\bf 
1}_{\M_h}+\big(n^{k+1}-
h^2\sum_{i,j=1}^{N_h}n_{i,j}^{k+1}{\bf 
1}_{\M_h}\big)\\
v^{k+1}
\end{pmatrix}\nonumber\\
\begin{pmatrix}
\bar{m}^{k+1}\\
\bar{w}^{k+1}
\end{pmatrix}
&=
\begin{pmatrix}
{m}^{k+1}+\theta(n^{k+1}-m^k)\\
{w}^{k+1}+\theta(v^{k+1}-w^k)
\end{pmatrix}.
\end{align}
As opposite to previous algorithms, this method does not need to 
invert any matrix. Its performance is explored in 
Section~\ref{sec:num}.
%

\subsection{Computing the proximity operator of $\varphi$}\label{proxvarphisection}

In each of the three last methods proposed in this section it is
important to compute $\prox_{\gamma\varphi}$ efficiently for 
each $\gamma>0$. 
In order to compute the proximity operator of the objective 
function 
in $(P_h^{d})$ and $(P_h)$, we need to 
introduce some notations and properties. Let $F\colon 
\RR_+\to\RR$ be a 
convex function which is differentiable in $\RR_{++}:=\{x\in \RR 
\; ; \; x>0\}$ and extended by 
taking the value $+\infty$ in $\RR_{--}:=\RR_{-}\setminus 
\{0\}$, let 
$d>0$,
let $\gamma>0$, 
and 
let $q\in \left]1,\pinf\right[$.  
For every 
$(m,w)\in\RR\times\RR^4$, define $F'(0):=\lim_{h\to 
0^+}(F(h)-F(0))/h$ which is assumed to exist in 
$[-\infty,\pinf[$, set
\begin{align}
\label{e:polinomio}
D(m)&=\menge{(p,\delta)\in \RR_+\times \RR}{p+\gamma 
F'(p)+\delta\geq m}
\end{align}
and, for every $(p,\delta)\in D(m)$, set
\small
\be\label{Qmw}
Q_{m,w}(p,\delta)=\left(p+\gamma 
F'(p)-m+\delta \right)
\left(p +\gamma^{2/q}
(q')^{1-2/q}(p+\gamma F'(p)-m+\delta)^{1-2/q}\right)^q-
\frac{\gamma}{q'}|P_Kw|^q.
\ee
\normalsize
Note that since $F'$ is increasing, given $p\in \RR_{+}$ for all 
$p'\geq p$ and $\delta'\geq m-p-\gamma F'(p)$ we have that 
$(p',\delta')\in D(m)$. Analogously, given $\delta \in \RR$ for 
all $\delta'\geq \delta$ and $p'\geq \prox_{\gamma F}(m-\delta) 
$ we have that $(p',\delta')\in D(m)$. Therefore, the following 
result is a direct consequence of the definition.
\begin{lemma} 
\label{l:properties}
Let $(m,w)\in\RR\times\RR^4$ and $(p,\delta)\in \RR_+\times \RR$ 
such that $F'(p)\in \RR$. We have that $Q_{m,w}(\cdot,\delta)$ 
and $Q_{m,w}(p,\cdot)$ 
are continuous and strictly 
increasing in $[\prox_{\gamma F}(m-\delta), +\infty[$ and 
$[m-p-\gamma F'(p),+\infty[$, respectively. Moreover, 
$$\lim_{\delta\to\pinf}Q_{m,w}(p,\delta)=\pinf \hspace{0.3cm} 
\mbox{and } \; 
\hspace{0.2cm}\lim_{p\to\pinf}Q_{m,w}(p,\delta)=\pinf.$$
\end{lemma} \smallskip

The following result provides the computation of the proximity 
operator of the objective function in $(P_{h}^{d})$. 

\begin{proposition}
\label{p:calprox}
Let $q\in\left]1,\pinf\right[$, let $d>0$, let $\gamma>0$, and 
let $F\colon\RR_{++}\to\RR$ be a convex differentiable function 
satisfying that 
$F'(0):=\lim_{h\to 0^+}(F(h)-F(0))/h\in\RR$ exists. Define 
  $\varphi\colon(m,w)\mapsto 
F(m)+\hat{b}(m,w)+\iota_{[0,d]}(m)$. Then, $\varphi$ is  
convex, proper, and lower semicontinuous. Moreover, given 
$(m,w)\in\RR\times\RR^4$ and $(p,\delta)\in D(m)$, by setting 
\begin{equation}
\label{e:defv}
v_{m,w}(p,\delta)=
\displaystyle{\frac{p}{p+\gamma^{2/q}
(q')^{1-2/q}(p+\gamma F'(p)-m+\delta)^{1-2/q}}}P_Kw,
\end{equation} 
we have that 
\begin{equation}
\label{e:proxfinal}
\prox_{\gamma\varphi}\colon(m,w)\mapsto
\begin{cases}
(0,0),\quad&\text{if }m\leq\gamma 
F'(0)\:\:\text{and}\:\:Q_{m,w}(0,0)\geq0,\\
(p^*,v_{m,w}(p^*,0)),&\text{if }
m\leq\gamma 
F'(0)\:\:\text{and}\:\:Q_{m,w}(0,0)<0< Q_{m,w}(d,0),\\
&\text{or}\:\: 
\gamma 
F'(0)<m< d+\gamma 
F'(d)\:\:\text{and}\:\:\:Q_{m,w}(d,0)>0,\\
(d,v_{m,w}(d,\delta^*)),&\text{otherwise},
\end{cases}
\end{equation}
where $p^*\geq 0$ and $\delta^*\geq0$ are the unique solutions 
to 
$Q_{m,w}(p,0)=0$ and $Q_{m,w}(d,\delta)=0$, respectively.
\end{proposition}

\begin{proof} Since the first assertion is clear from 
Lemma~\ref{l:partial}, we only prove \eqref{e:proxfinal}.
Let $(p,v)$ and $(m,w)$ in $\RR\times\RR^4$ such that 
$(p,v)=\prox_{\gamma\varphi}(m,w)$. 
It follows from \eqref{e:prox} that 
$(p,v)\in\dom\partial\varphi\subset]0,d]\times K\cup\{(0,0)\}$ 
and
\begin{equation}
\label{e:calc1}
(m-p,w-v)\in\gamma\partial\varphi(p,v).
\end{equation}
Since the solution of the previous inclusion is unique in terms 
of $(p,v)$, it is enough to check that 
\eqref{e:proxfinal} satisfies \eqref{e:calc1} for each case.

First note that $\gamma F'(0)\geq m$ and $m\leq d+\gamma F'(d)$ 
imply $(0,0)\in D(m)$ and $(d,0)\in D(m)$, respectively. We 
split our 
proof in three cases: $m\leq \gamma F'(0)$, $\gamma F'(0)<m\leq 
d+\gamma F'(d)$, and $m>d+\gamma F'(d)$. 
\vskip .3cm
$\bullet$ \underline{Case $m\leq \gamma F'(0)$:}
First suppose that 
 $Q_{m,w}(0,0)\geq 0$.
We have
\begin{align}
Q_{m,w}(0,0)\geq 0\quad
&\Leftrightarrow\quad 
(\gamma F'(0)-m)^{q-1}\geq 
\frac{1}{\gamma (q')^{q-1}}|P_Kw|^q,\nonumber\\
&\Leftrightarrow\quad 
\gamma F'(0)-m\geq 
\frac{\gamma^{1-q'}}{q'}|P_Kw|^{q'},
\end{align}
which, from Lemma~\ref{l:partial}, is equivalent to $(m-\gamma 
F'(0),w)\in\gamma \partial\hat{b}(0,0)$. Therefore, since 
$0\in \partial \iota_{[0,d]}(0)= \RR_{-}$, we obtain   
$(m-0,w-0)\in\gamma\partial\hat{b}(0,0)+\gamma 
\{F'(0)\}\times\{0\}+\gamma\partial 
\iota_{[0,d]}(0)\times\{0\}\subseteq \gamma\partial\varphi(0,0)$ 
and, hence, \eqref{e:calc1} holds with $p=v=0$.
Now suppose that $Q_{m,w}(0,0)<0< Q_{m,w}(d,0)$. 
Lemma~\ref{l:properties} ensures the existence and uniqueness of 
a strictly positive solution in $]0,d[$ to $Q_{m,w}(\cdot,0)=0$, 
which is 
called 
$p^*$. Let us set
\be\label{defvstar1}v^*:=v_{m,w}(p^*,0)=\displaystyle{\frac{p^*}{p^*+\gamma^{2/q}
(q')^{1-2/q}(p^*+\gamma F'(p^*)-m)^{1-2/q}}}P_Kw\ee
and let us prove that 
$(p^*,v^*)$ satisfies \eqref{e:calc1}. Indeed, since 
$p^*+\gamma F'(p^*)-m>\gamma F'(0)-m\geq 0$, $Q_{m,w}(p^*,0)=0$ 
is equivalent to 
\begin{equation}
\label{e:auxprox1}
m-p^*-\gamma F'(p^*)=-
\frac{\gamma}{q'}\left(\frac{|P_Kw|}{p^*+\gamma^{2/q}
	(q')^{1-2/q}(p^*+\gamma F'(p^*)-m)^{1-2/q}}\right)^q=-
	\frac{\gamma}{q'}\frac{|v^*|^q}{p^{*q}}
\end{equation}
and since $v^*\in K$ we have from  \eqref{defvstar1} that
\begin{align}
\label{e:auxprox2}
w&\in\displaystyle{\frac{p^*+\gamma^{2/q}
		(q')^{1-2/q}(p^*+\gamma 
		F'(p^*)-m)^{1-2/q}}{p^*}}v^*+N_K(v^*),\nonumber\\
\Leftrightarrow \quad 
w-v^*&\in\displaystyle{\frac{\gamma^{2/q}
		(q')^{1-2/q}(p^*+\gamma 
		F'(p^*)-m)^{1-2/q}}{p^*}}v^*+N_K(v^*),\nonumber\\
\Leftrightarrow\quad 
w-v^*&\in\gamma\frac{|v^*|^{q-2}v^*}{p^{*q-1}}+N_K(v^*),
\end{align}
where the last line follows from \eqref{e:auxprox1} and 
straightforward computations. Therefore, since 
$p^*\in]0,d[$, $N_{[0,d]}(p^*)=\{0\}$ and  from 
\eqref{e:auxprox1}, \eqref{e:auxprox2}, and  
Lemma~\ref{l:partial} we obtain
$$(m-p^*,w-v^*)\in \gamma\partial\hat{b}(p^*,v^*)+\gamma 
F'(p^*)\times\{0\}\subseteq \gamma\partial\varphi(p^*,v^*)$$ and 
\eqref{e:calc1} 
follows. Now suppose that $Q_{m,w}(d,0)\leq0$. Then, from 
Lemma~\ref{l:properties} there exists a unique $\delta^*\geq0$ 
such 
that 
$Q_{m,w}(d,\delta^*)=0$. Let us set $p^*=d$ and 
$v^*=v_{m,w}(d,\delta^*)$. 
In this case, $Q_{m,w}(d,\delta^*)=0$ is equivalent to 
\be
\label{e:aux2} 
m-d-\gamma 
F'(d)-\delta^*=-
\frac{\gamma}{q'}\left(\frac{|P_Kw|}{d+\gamma^{2/q}
	(q')^{1-2/q}(d+\gamma F'(d)-m+\delta^*)^{1-2/q}}\right)^q=-
\frac{\gamma}{q'}\frac{|v^*|^q}{d^{q}}
\ee
and since $\delta^*\in  N_{[0,d]}(d)=\RR_{+}$ as 
before we deduce 
\be
\label{e:aux2proxy} 
\displaystyle{m-d\in-
	\frac{\gamma}{q'}\frac{|v^*|^q}{d^{q}}+\gamma 
	F'(d)}+N_{[0,d]}(d).
\ee
On the other hand, by arguing analogously as in 
\eqref{e:auxprox2} we obtain
\begin{equation}
\label{e:calc22}
\displaystyle{w-v^*\in\gamma\frac{|v^*|^{q-2}v^*}{d^{q-1}}+N_K(v^*)}.
\end{equation}
Hence, from \eqref{e:aux2}, \eqref{e:aux2proxy}, and 
Lemma~\ref{l:partial} we obtain that \eqref{e:calc1} holds with 
$(p,v)=(d,v^*)$.
\vskip .3cm
$\bullet$ \underline{Case $\gamma F'(0)<m<d+\gamma F'(d)$:}
The difference with respect to the previous case is that 
$Q_{m,w}$ 
may be not defined at $(0,0)$. However, since $F$ is convex and 
lower 
semicontinuous, there exists $\prox_{\gamma F}m$, which is the 
unique solution of $z+\gamma 
F'(z)=m$. Thus, in this case, using that $F'$ is increasing, we 
get that $\prox_{\gamma F}m \in ]0,d[$ and that 
$Q_{m,w}(\prox_{\gamma F}m,0)=-\gamma|P_Kw|^q/q'\leq0$.
Now suppose that $Q_{m,w}(d,0)>0$. Then Lemma~\ref{l:properties} 
provides the existence of $p^*\in[\prox_{\gamma 
F}m,d[\subset]0,d[$ such that 
$Q_{m,w}(p^*,0)=0$. The verification of \eqref{e:calc1} for 
$(p^*,v(p^*,0))$
is analogous to the previous case since $p^*>0$. Otherwise, if 
$Q_{m,w}(d,0)\leq 0$, there exists a unique $\delta^*\geq0$ 
such that $Q(d,\delta^*)=0$ and, by setting $v^*=v(d,\delta^*)$ 
we 
can repeat the computation in \eqref{e:aux2} and the result 
follows.
\vskip .3cm
$\bullet$ \underline{Case $m\geq d+\gamma F'(d)$:} Defining
$\hat{\delta}=m-d-\gamma F'(d)\geq 0$, we have 
$Q_{m,w}(d,\hat{\delta})=
-\gamma|P_Kw|^q/q'\leq0$ and $(d,\delta)\in D(m)$ for every 
$\delta\geq \hat{\delta}$. Therefore, as before, there exists a 
unique $\delta^*\geq\hat{\delta}$ such that 
$Q_{m,w}(d,\delta^*)=0$. By setting $v^*=v(d,\delta^*)$ the 
result 
follows as in  the previous cases.
\end{proof}

In the absence of upper bound constraints for the $m$ the 
computations are simpler. We provide 
this simplified version for solving $(P_h)$ in the following 
corollary, whose proof 
is analogous to the proof of Proposition \ref{p:calprox} and so 
we omit it. Formally, the result can be seen as a limit case of 
Proposition \ref{p:calprox} when $d\to +\infty$.

\begin{corollary}
\label{p:calproxb}
Let $q\in\left]1,\pinf\right[$, let $\gamma>0$  and suppose that 
$F'(0):=\lim_{h\to 0^+}(F(h)-F(0))/h\in\RR$ exists. 
Moreover, set  $\varphi\colon(m,w)\mapsto F(m)+\hat{b}(m,w)$.  
Then, $\varphi$ is 
convex, proper, and lower semicontinuous and 
\begin{equation}
\label{e:proxfinalb}
\prox_{\gamma\varphi}\colon(m,w)\mapsto
\begin{cases}
(0,0),\quad&\text{if }m\leq\gamma 
F'(0)\:\:\text{and}\:\:Q_{m,w}(0,0)\geq0;\\
(p^*,v_{m,w}(p^*,0)),&\text{otherwise},
\end{cases}
\end{equation}
where $p^*\geq 0$ is the unique solution to 
$Q_{m,w}(p,0)=0$ and $v_{m,w}$ is defined in \eqref{e:defv}.
\end{corollary}

\begin{remark}
	Note that, in the particular case when $q=2$, $F'\equiv0$, 
	and 
	$K=\RR^4$, the computation of the proximity operator in 
	Corollary~\ref{p:calproxb} reduces 
	to 
	that of {\rm\cite{Peyre}}.
\end{remark}
Another important case to be considered is when the function 
$F'$ satisfies \eqref{derivadamenosinfty}, in which case the 
computation is also simpler.
Since the proofs can be derived from the proof of 
Proposition~\ref{p:calprox} we omit them.
\begin{corollary}
	\label{c:finalprox}
Let $q\in\left]1,\pinf\right[$, let $d>0$, let $\gamma>0$ and 
suppose that 
$F'(0):=\lim_{h\to 0^+}(F(h)-F(0))/h\ =-\infty$. Define 
  $\varphi\colon(m,w)\mapsto 
F(m)+\hat{b}(m,w)+\iota_{[0,d]}(m)$.  
Then, $\varphi$ is 
convex, proper, lower semicontinuous, and 
\begin{equation}
\label{e:proxfinalc}
\prox_{\gamma\varphi}\colon(m,w)\mapsto
\begin{cases}
(p^*,v_{m,w}(p^*,0)),&\text{if }
m< d+\gamma 
F'(d)\:\:\text{and}\:\:\:Q_{m,w}(d,0)>0;\\
(d,v_{m,w}(d,\delta^*)),&\text{otherwise},
\end{cases}
\end{equation}
where $p^*> 0$ and $\delta^*\geq0$ are the unique solutions to 
$Q_{m,w}(p,0)=0$ and $Q_{m,w}(d,\delta)=0$, respectively, and 
$v_{m,w}$ is defined in \eqref{e:defv}. On the other hand, 
defining $\phi\colon(m,w)\mapsto 
F(m)+\hat{b}(m,w)$, we have that 
\begin{equation}
\label{e:proxfinald}
\prox_{\gamma\phi}\colon(m,w)\mapsto (p^*,v_{m,w}(p^*,0)), 
\end{equation}
where $p^*> 0$ is the unique solution of $Q_{m,w}(p,0)=0$.
\end{corollary}

\begin{remark}
Note that in Corollary~\ref{c:finalprox} we ensure that the 
strict positivity of the first coordinate of the proximal 
mapping associated to the objective function, which cannot be 
guaranteed neither in Proposition \ref{p:calprox} nor Corollary 
\ref{p:calproxb}. 
\end{remark}

\begin{remark}
Since Proposition~\ref{p:calprox} gives a closed expression for 
$\prox_{\gamma\varphi}$, an advantage of the last three methods 
in Section~\ref{sec:5.meth}
respect to ADMM in our setting, is that each step of the 
algorithm is 
computable for every differentiable function $F(x_{i,j},\cdot)$. 
\end{remark}

\section{Numerical experiments}
\label{sec:num}
In the following, we present numerical tests aiming at illustrating the different features of the proposed schemes as well as assessing both their performance and accuracy in the setting of the MFG system \eqref{mfgestacionario}. For the sake of simplicity, we shall use the following abbreviations to refer to the implemented algorithms.
\begin{itemize}
\item {\bf ADMM}: Alternating direction method of multipliers, as in Section \ref{admm}.
\item {\bf CP-U}: Chambolle-Pock algorithm without splitting, as in Section \ref{cpu}.
\item {\bf PCPM-U}:  Predictor-corrector proximal multiplier method without splitting, as in Section \ref{pcpmu}.
\item {\bf MS-U}: Monotone+skew without splitting, as in Section \ref{msu}.
\item {\bf CP-SP}: Chambolle-Pock algorithm with splitting and projected on the mass constraint, as in Section \ref{cpsp}.
\item {\bf MS-SP}: Monotone+skew with splitting and projected on the mass constraint.
\item {\bf PCPM-SP}: Predictor-corrector proximal multiplier method with splitting and projected on the mass constraint.
\end{itemize}
\paragraph{Implementation and parametric choices.}
The starting point of our numerical implementation is the finite difference discretization presented in section \ref{prelim}. Once the discretized operators have been assembled, we proceed to implement the optimization algorithms derived in section \ref{als}. We highlight the simplicity of the proposed methods, as the inner loops only requires the solution of nonlinear scalar equations and matrix inversions. However, if the number of degrees of freedom increases, as in the time-dependent MFG setting, one needs to resort to preconditioning, as already discussed in \cite{MR2928376,Andreev16}. However, the \textbf{-SP} versions of the algorithms, i.e., with splitting and projection on the mass constraint, do not  require matrix inversion.  The nonlinear equations related to the proximal operator calculation are solved separately for every gridpoint based on sequential information, and therefore they are fully parallelizable, a property which we exploit in our code.
Each optimization algorithm presented in section \ref{als} has a set of parameters to be set offline. Our choice of parameters falls within  the prescribed parametric bounds guaranteeing convergence. For instance in Theorem \ref{t:CPockproj}, the Chambolle-Pock algorithm requires $\gamma\tau\|\Xi\|^2<1$. Although we observe that choices of $\gamma\,,\tau$ violating this condition can lead to faster convergence, the accuracy and stability of the algorithm deteriorates. 
\noindent The optimization routines are stopped when the norm of the difference between the primal variables of two consecutive iterations  has reached the threshold $\|(m^{n+1},w^{n+1})-(m^n,w^n)\|\leq\frac15h^3\,,$ where $h$ is the mesh parameter of the finite difference approximation, in order to ensure that the numerical error of the discretization does not interfere with the stopping rule of the iterative loop.

\paragraph{Test 1: assessing accuracy and convergence.} In order to assess the accuracy and performance of the proposed algorithms we study a first test case proposed in \cite{gomesex}. We consider the first-order stationary MFG system
$$\ba{rcl}
	 \half |\nabla u |^2 -\lambda &=&  \log m - \sin(2\pi x)- \sin (2\pi y), \\[4pt]
	 \mbox{div}(m\nabla u)&=& 0,  \; \; \; \int_{\TT^2} m \dd x=1, \; \; \; \int_{\TT^2} u \dd x=0, \ea$$
with explicit solution 
\begin{equation}\label{exactat1}
u(x,y)= 0, \; \; \; m(x,y)= e^{\sin(2\pi x)+\sin (2\pi y)-\lambda}\,,\;\;\;\lambda=\log\left(\int_{\TT^{2}}e^{\sin(2\pi x)+\sin (2\pi y)} \dd x \dd y\right)\,.
\end{equation}
\begin{figure}[!ht]
\centering
\includegraphics[width=0.49\textwidth]{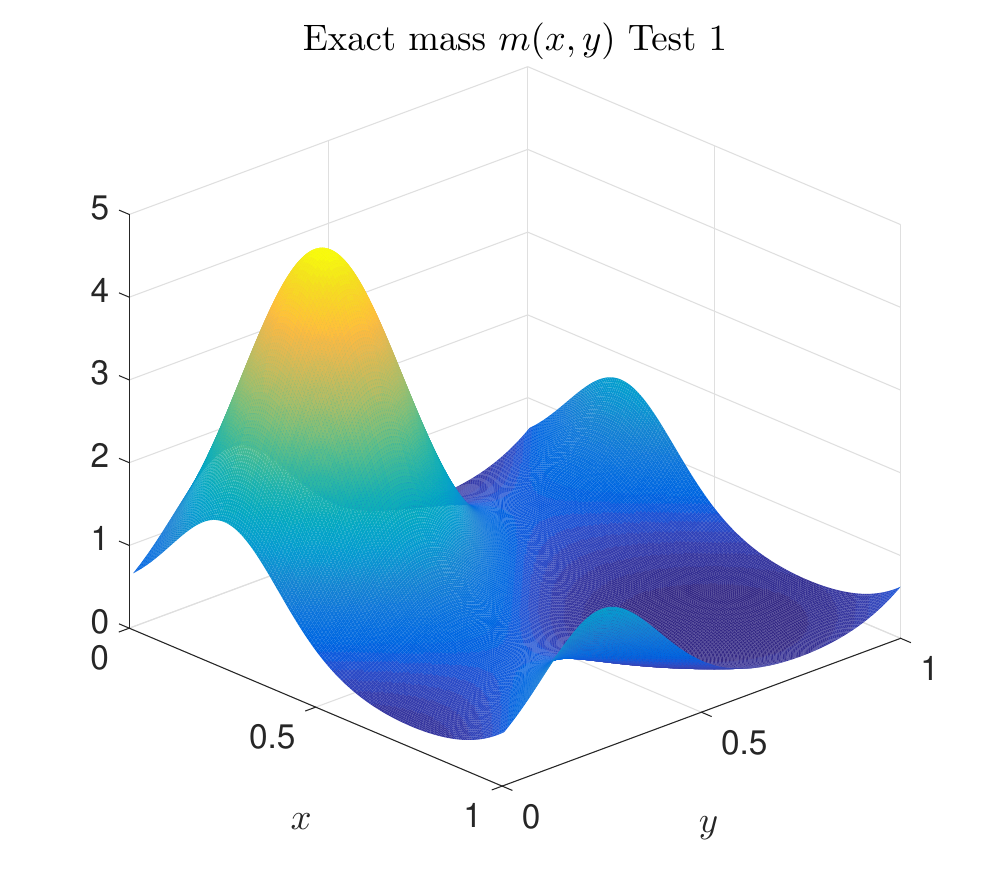}
\includegraphics[width=0.49\textwidth]{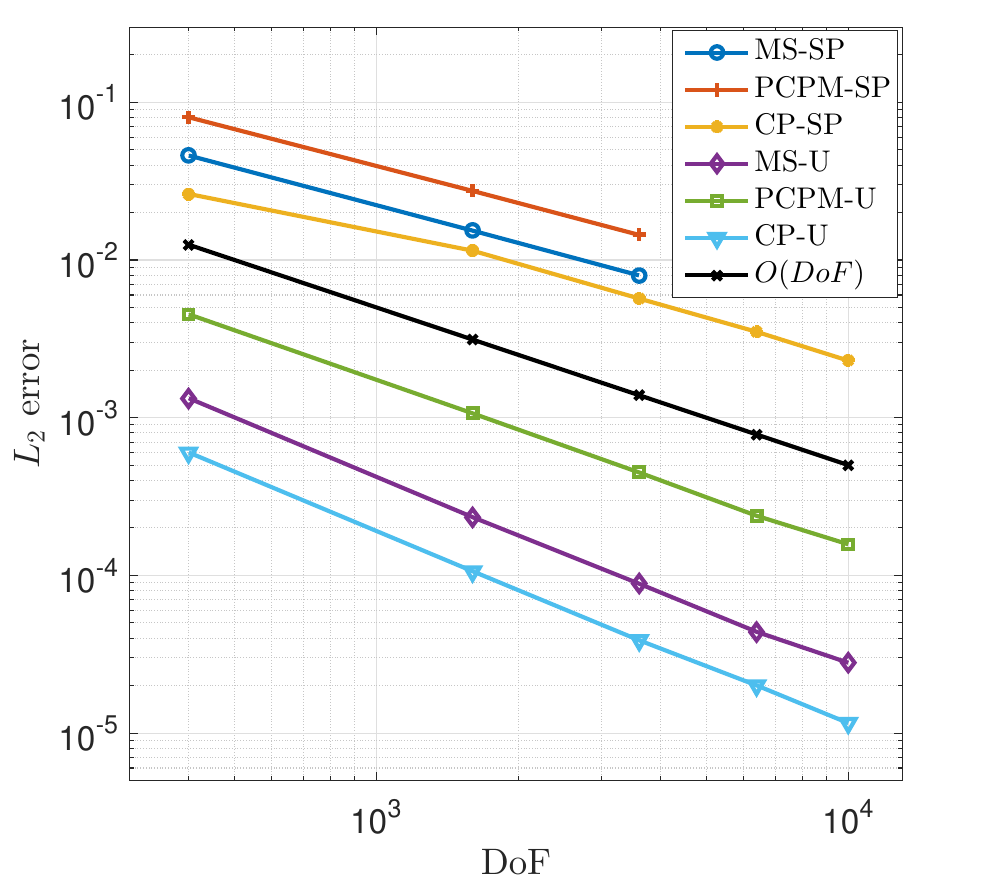}
\vskip 5mm
\includegraphics[width=0.49\textwidth]{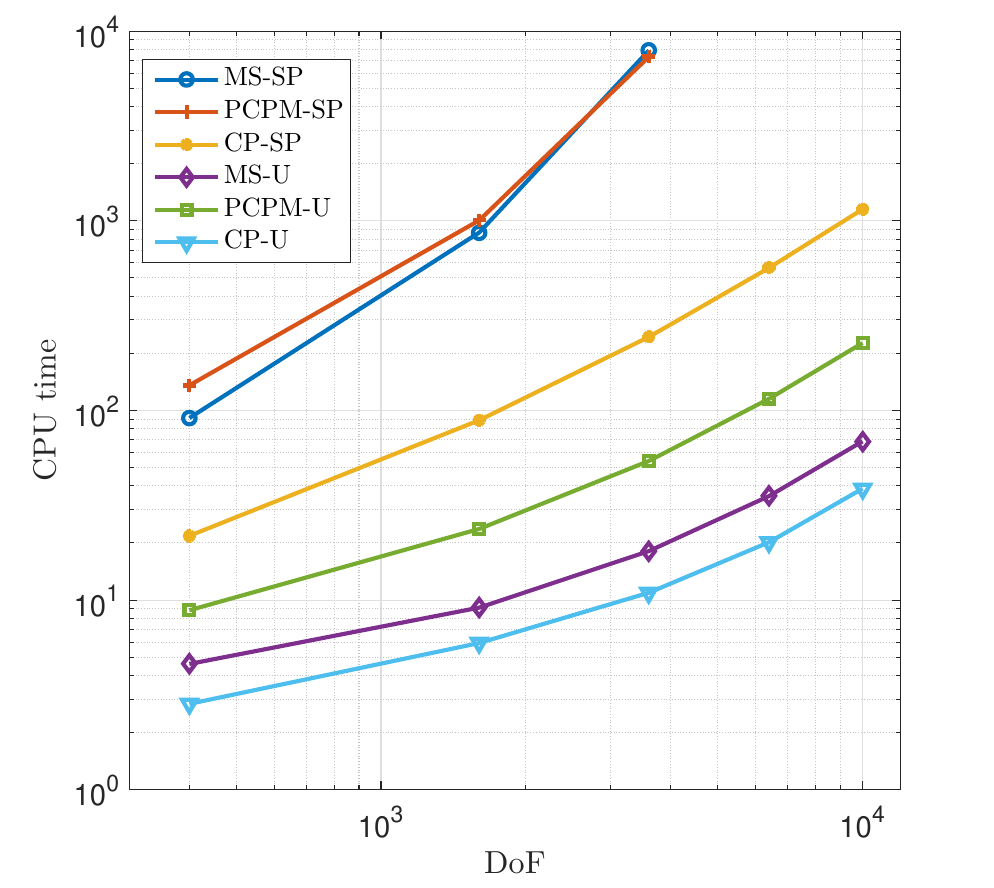}
\includegraphics[width=0.49\textwidth]{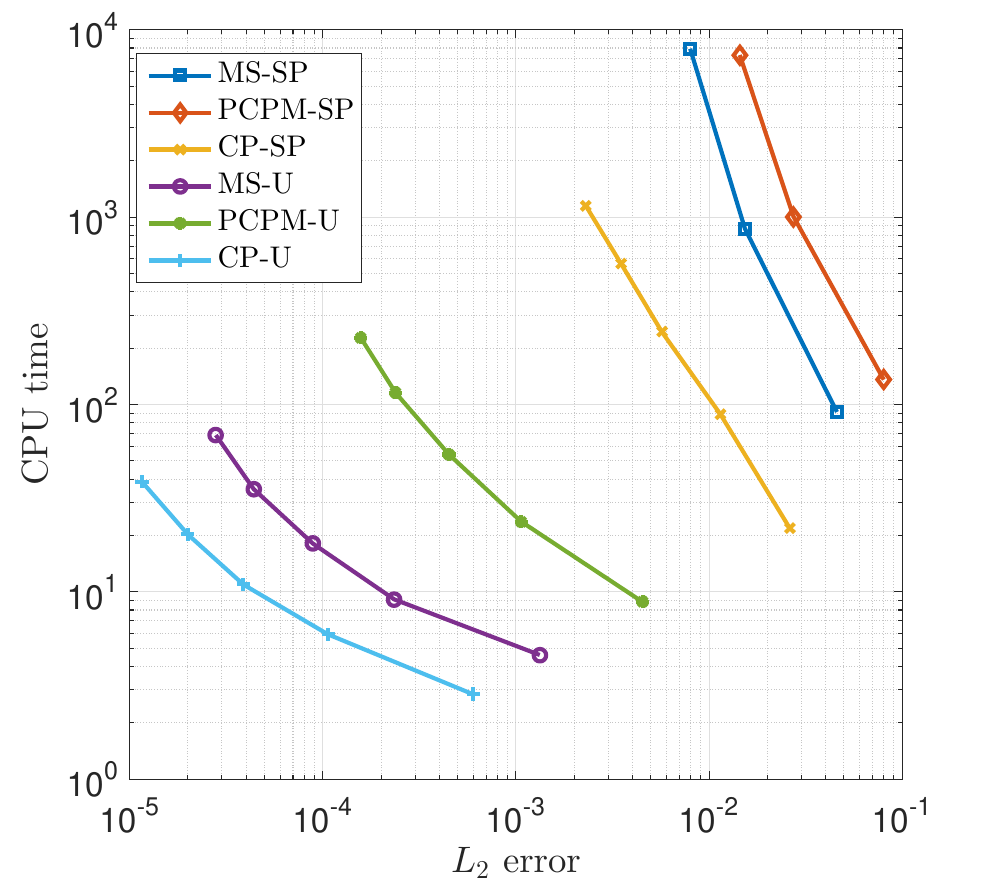}
\caption{ Test 1. Top left: exact mass $m(x,y)$ as in \cite{gomesex}. Top right: convergence rates for the proposed schemes, first-order convergence with respect to the number of nodes is achieved in all the cases. Bottom left: performance plots, degrees of freedom vs. CPU time, the Chambolle-Pock is the fastest algorithm for a fixed mesh parameter $h$. Bottom right: efficiency plots, error vs. CPU time, the Chambolle-Pock algorithm is consistently the most efficient implementation.}\label{fig:test1}
\end{figure}

\noindent In this test, we study the behavior of all the proposed algorithms, for different discretization parameters $h$ and the related number of degrees of freedom $DoF=1/h^2$, both in their unsplit and split versions. Results presented in Figure \ref{fig:test1} indicate that although all the algorithms achieve the same convergence rate in $h$, measured in the $L^2$ norm between the last discrete iteration and the exact solution \eqref{exactat1}, the unsplit versions have smaller error. More importantly, when comparing CPU time (or number of iterations) against $L^2$ errors, unsplit algorithms perform considerably better. However, split algorithms are still competitive and provide a reliable way to approximate the solution without performing any matrix inversion. Overall, the Chambolle-Pock algorithm exhibits the best performance and accuracy in both unsplit and split versions. We shall stick to this choice in the following tests.

\paragraph{Test 2: comparing with the ADMM algorithm.}
This second text is based on the recent work by \cite{benamoucarlier15}, where an implementation of the ADMM algorithm is presented for MFG and optimal transportation problems. We compare the performance of the ADMM and the CP-U methods for different discretization parameters and viscosity values $\nu$. For this, the system \eqref{mfgestacionario} is cast with ${\color{red}F}(x,y,m)=\frac12(m-\bar{m}(x,y))^2$ and $q=2$, where $\bar{m}(x,y)$ is a Gaussian profile as depicted in Figure \ref{fig:test2}. In the case $\nu=0$, since our reference $\bar{m}$ is already of mass equal to 1, the exact solution for this problem is given by $u\equiv 0$, $\lambda=0$, and $m=\bar{m}(x)$, and a convergence analysis with respect to this solution is presented in Table \ref{tabtest21}. From this same table, it can be seen that for different discretization parameters the CP-U algorithm converges to solutions of the same accuracy in a reduced number of iterations. For a fixed discretization, and with varying small viscosity values, the same conclusion is reached in Figure \ref{fig:test2}. However, as viscosity increases, the ADMM algorithm yields faster computation times than the CP-U implementation (see Table \ref{tabtest22}). 

\begin{figure}[!ht]
\centering
\includegraphics[width=0.49\textwidth]{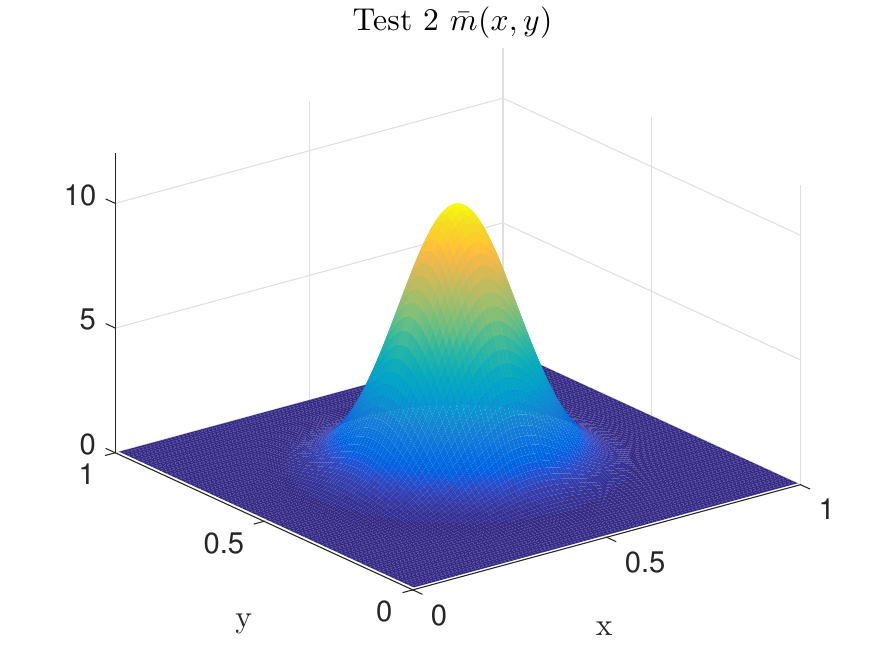}
\includegraphics[width=0.49\textwidth]{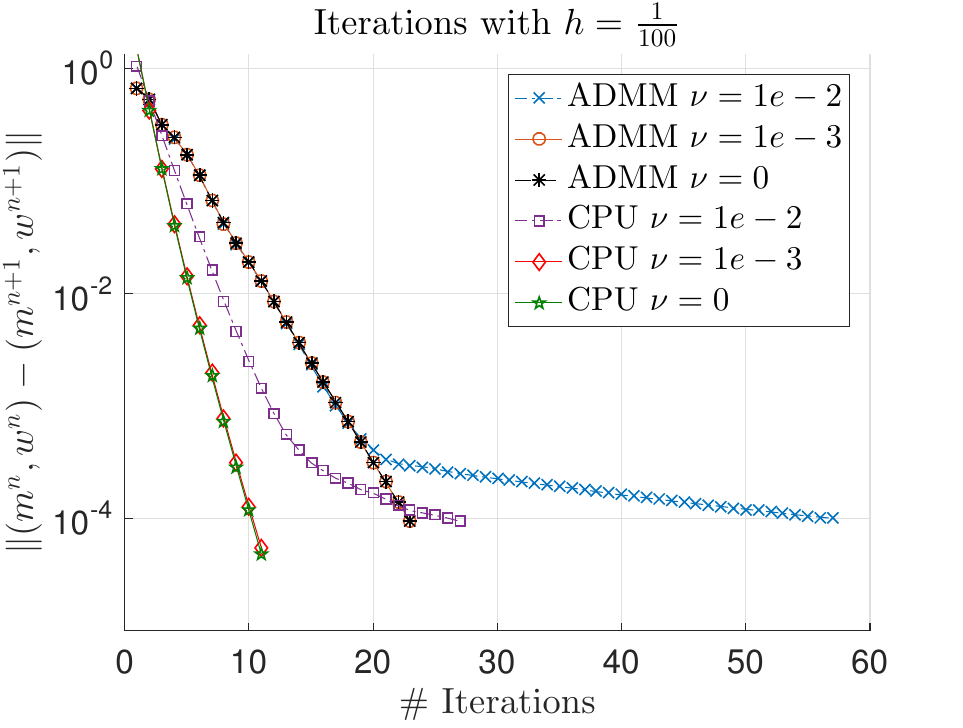}

\caption{Test 2. Left: reference mass $\bar{m}(x,y)$. Right: iterative behavior of different schemes, for mesh parameter $h=\frac{1}{100}$ and different values of $\nu$. The unsplit Chambolle-Pock algorithm (CP-U) outperforms the ADMM algorithm for low values of $\nu$. }\label{fig:test2}
\end{figure}

\begin{table}[h]
\centering
\setlength{\tabcolsep}{1mm}
  \begin{tabular}{cp{1.7cm}p{1.9cm}p{1.7cm}p{1.7cm}p{1.9cm}p{1.7cm}}
  \hline\\
  & \multicolumn{3}{c}{ADMM}&  \multicolumn{3}{c}{CP-U} \\
\cmidrule(lr){2-4} \cmidrule(lr){5-7}\\
   DoF& Time&   Iterations & $L_2$ error& Time&   Iterations & $L_2$ error\\
   \cmidrule(lr){1-1}  \cmidrule(lr){2-4} \cmidrule(lr){5-7}\\
    $20^2$   &1.6 [s] &    15 &   5.42E-4 &0.4 [s]&    4 &   1.10E-4 \\
    $40^2$ &3.7 [s]&    19 &  8.44E-5  &0.9 [s] &    6 &  9.44E-5\\
    $60^2$&21.2 [s] & 21 &   8.16E-5 &7.0 [s] &   8 &  9.15E-5\\
    $80^2$&33.2 [s] & 22 &   7.92E-5 &10.2 [s] &   9 &  8.99E-5\\
    $100^2$&87.41 [s] & 23 &   7.35E-5 &30.3 [s] &   11 &  7.04E-5\\
    \hline 
  \end{tabular}
  \caption{Test 2. Different tests with varying number of grid nodes (DoF). Case with $\nu=0$, exact solution $m=\bar{m}(x,y)$. For a similar accuracy, the CP-U algorithm has a reduced number of iterations in comparison to the ADMM routine.}\label{tabtest21}
\end{table}
\begin{table}[h]
\centering
\setlength{\tabcolsep}{1mm}
  \begin{tabular}{ccccc}
  \hline\\
  & \multicolumn{2}{c}{ADMM}&  \multicolumn{2}{c}{CP-U} \\
\cmidrule(lr){2-3} \cmidrule(lr){4-5}\\
   $\nu$& Time&   Iterations & Time&   Iterations \\
   \cmidrule(lr){1-1}  \cmidrule(lr){2-3} \cmidrule(lr){4-5}\\
    1   &8.4 [s] &    16 &17.5 [s]&    46 \\
    0.1 &30.4[s]&    31 &65.6 [s] &    73\\
    1E-2&26.4 [s] & 27 &9.8 [s] &   11\\
    1E-3&21.3 [s] & 21 &7.3 [s] &   8\\
    0&21.2 [s] & 21 &7.0 [s] &   8\\
    \hline 
  \end{tabular}
  \caption{Test 2. Different tests with varying viscosity parameter $\nu$. Discretization parameter $h=\frac{1}{60}$. The ADMM algorithm performs better for higher viscosity values, the CP-U algorithm is consistently faster for low viscosities.}\label{tabtest22}
\end{table}

\newpage
\paragraph{Test 3: adding density constraints.}
The following test mimics the setting presented in \cite{AchdouCapuzzo10}, with $q=2$ and
$$
f(x,y,m)=m^2-\bar{H}(x,y)\,,\quad \bar{H}(x,y)=\sin(2\pi y) + \sin(2\pi x) + \cos(4\pi x)\,.
$$
The purpose of this test is twofold. First, in the unconstrained mass case, we reproduce the results presented in \cite{AchdouCapuzzo10} and in \cite{CaCa16}. As shown in Table \ref{tabtest31} (left), we recover the same values  for $\lambda$ reported in the aforementioned references. The CP-U algorithm performs consistently well for different viscosity values and reaches convergence after a reduced number of iterations. Computational times are comparable to those reported in \cite{CaCa16}, considering that the CP-U is a first order method.
Next, we perform similar tests but including an upper bound on the mass, 
$$m(x,y)\leq d(x,y):= \mathcal{I}_R(x,y)+(1-\mathcal{I}_R(x,y))\bar d\,,\quad \mathcal{I}_R(x,y):= \begin{cases} 
      1 & x^2+y^2\leq R^2 \\
      0 & \text{otherwise}
   \end{cases}\,,\quad \bar d=1.3,R=0.25\,.$$  Figure \ref{fig:test3} illustrate the effectiveness of our approach, as solutions vary from the unconstrained case in order to satisfy both the MFG system and the additional constraint. The inclusion of mass constraints generate plateau areas where the constraint is active. In Table \ref{tabtest31} (right), we observe that the scheme does not deteriorate its performance in the constrained formulation, leading to convergence in a similar number of iterations as in the unconstrained case.

\begin{table}[h]
\centering
\setlength{\tabcolsep}{1mm}
  \begin{tabular}{cp{1.7cm}p{1.9cm}p{1.7cm}p{1.7cm}p{1.9cm}}
  \hline\\
  & \multicolumn{3}{c}{Unconstrained mass}&  \multicolumn{2}{c}{Constrained mass $m\leq d$} \\
\cmidrule(lr){2-4} \cmidrule(lr){5-6}\\
   $\nu$& Time&   Iterations & $\lambda$& Time&   Iterations\\
   \cmidrule(lr){1-1}  \cmidrule(lr){2-4} \cmidrule(lr){5-6}\\
    1   &6.82 [s]  & 11 &   0.9786 &46.65 [s]&    51  \\
    0.1 &13.26 [s]  & 27 &  1.100  &13.81 [s] &    24\\
    1E-2&34.62 [s] & 78 &   1.1874 &29.09 [s] &   56\\
    1E-3&22.88 [s] & 84 &   1.1922 &27.87 [s] &   56\\
    \hline 
  \end{tabular}
  \caption{Test 3. Performance for the CP-U algorithm in \cite{AchdouCapuzzo10} with different viscosity parameter $\nu$, and upper bound on the mass, $m\leq d$. $f(x,y,m)=m^2-\bar H(x)$. Mesh parameter is set to $h=1/50$. The results for the unconstrained case are in accordance, in accuracy with the values  for $\lambda$ presented in \cite{CaCa16}. Our scheme performs robustly with respect to the viscosity parameter and the inclusion of mass constraints.}\label{tabtest31}
\end{table}

\begin{figure}[!ht]
\centering
\includegraphics[width=0.49\textwidth]{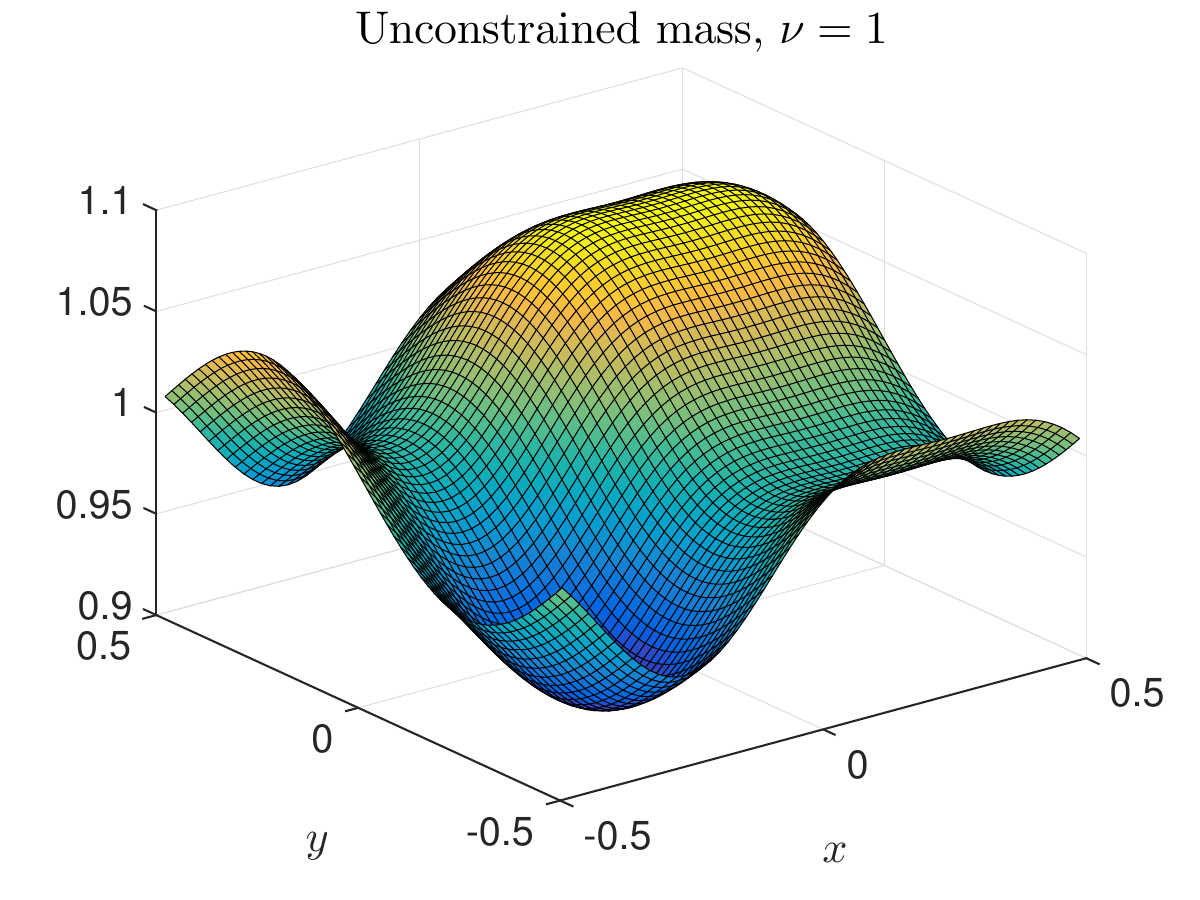}
\includegraphics[width=0.49\textwidth]{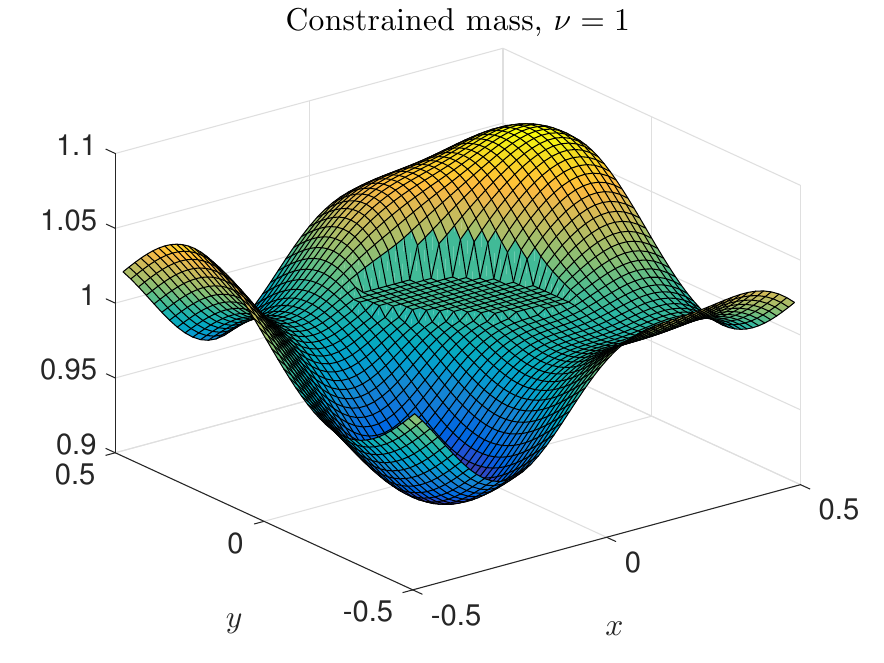}
\vskip 5mm
\includegraphics[width=0.49\textwidth]{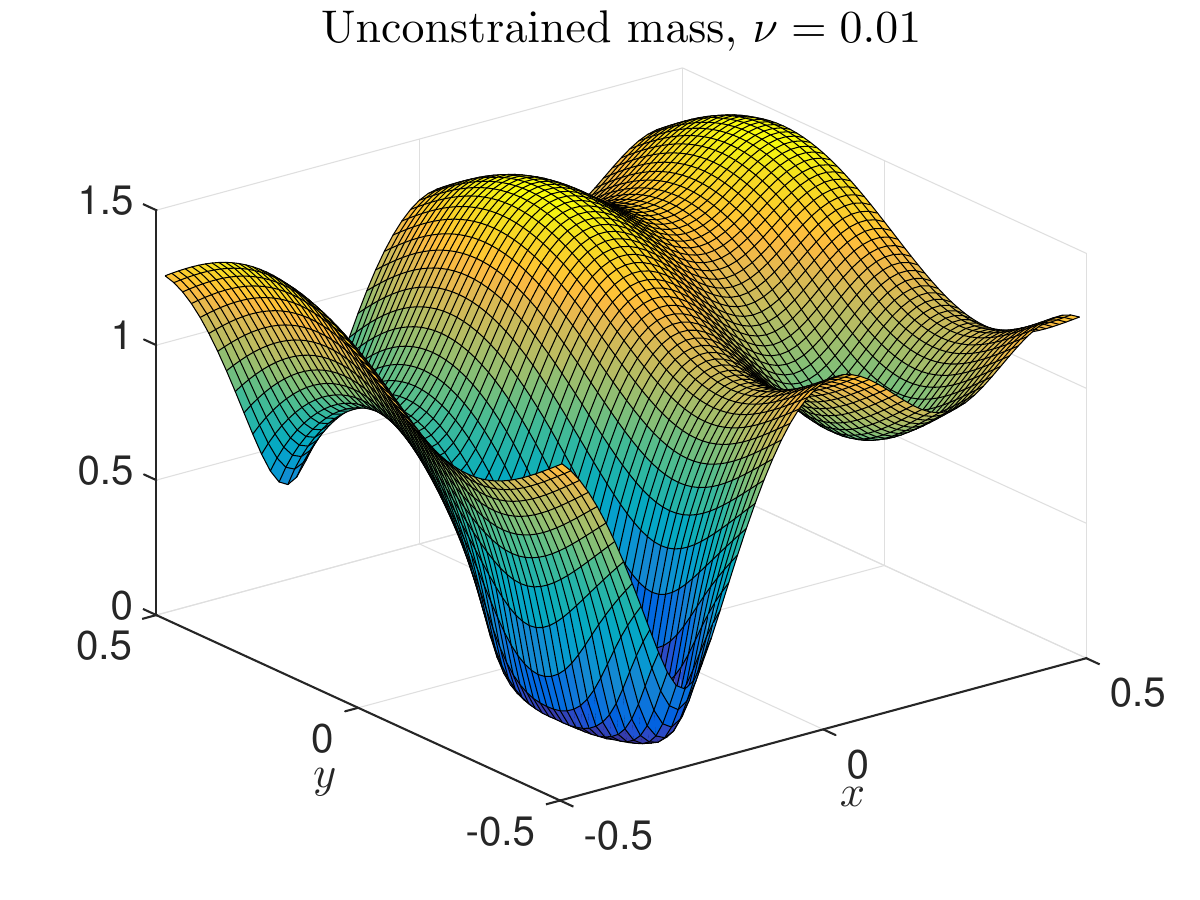}
\includegraphics[width=0.49\textwidth]{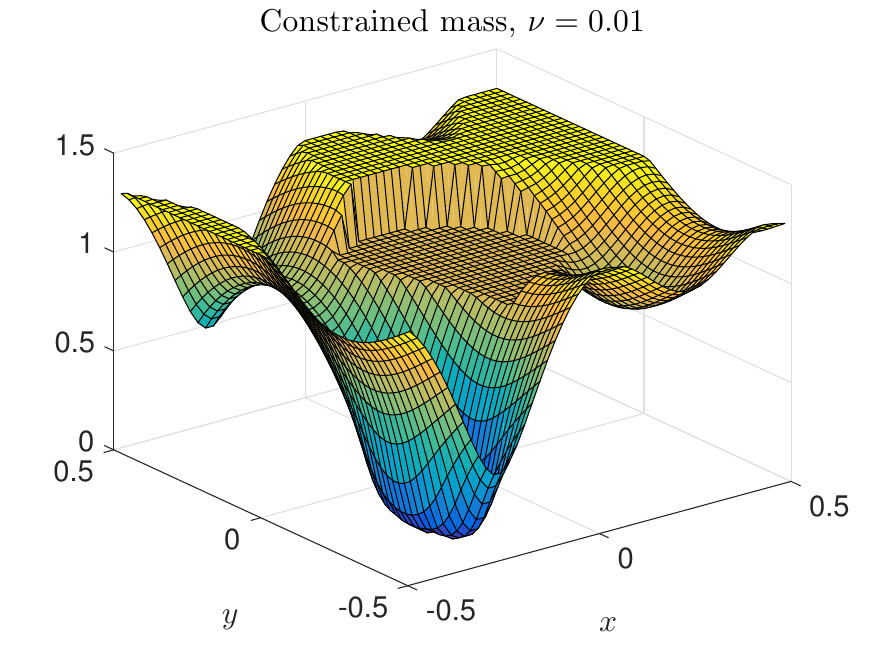}

\caption{Test 3. Left: unconstrained mass solutions for different viscosity parameters. Right: constrained mass solutions for different viscosity parameters. In both cases the upper bound is selected in order to be active, thus generating a plateau of constant mass.}\label{fig:test3}
\end{figure}

\paragraph{Test 4: MFG with $q\neq 2$.} In this last test, we further explore the versatility of the proposed framework by considering the same setting as in Test 3 in the unconstrained case with $\nu=1$, but with different values of  $q>1$. Results are presented in Figure \ref{fig:test4}. In general, it can be observed that the performance of the CP-U method remains unaltered, and solutions tend to be uniform when $q$ is close to 1, whereas increasing $q$ leads to sharper solutions with higher extremal values, as shown in Table \ref{tabtest41}.

\begin{figure}[!ht]
\centering
\includegraphics[width=0.49\textwidth]{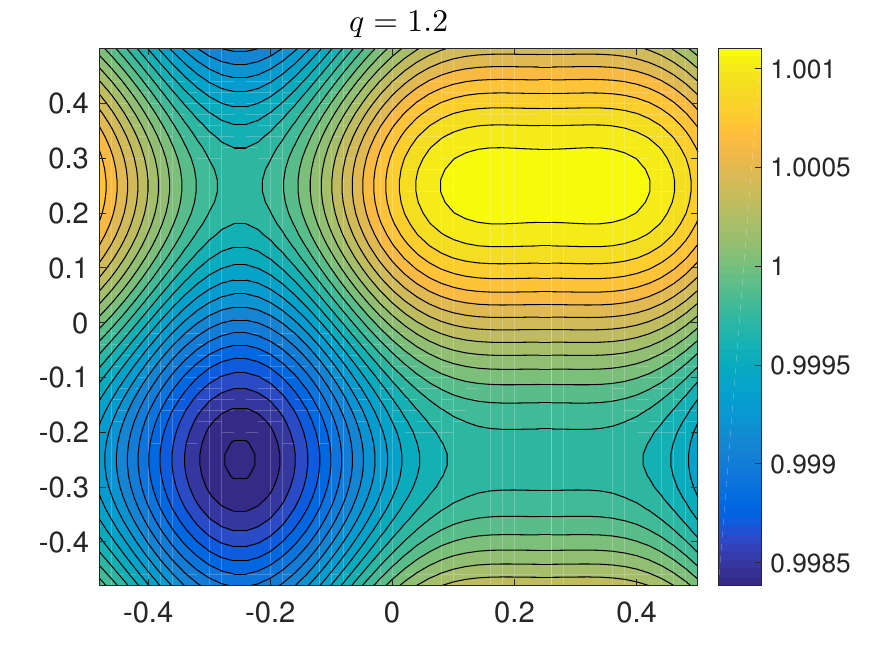}
\includegraphics[width=0.49\textwidth]{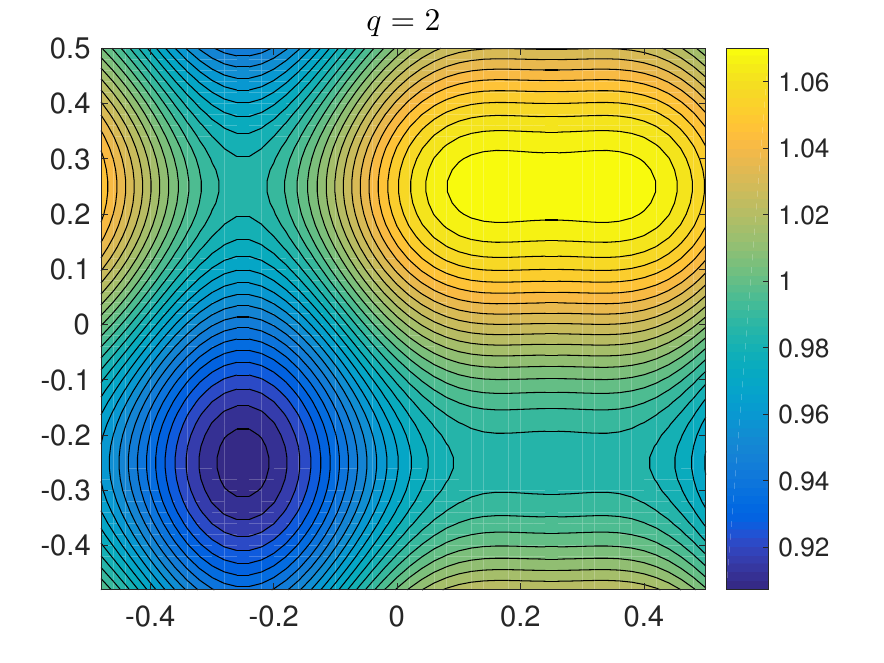}
\vskip 5mm
\includegraphics[width=0.49\textwidth]{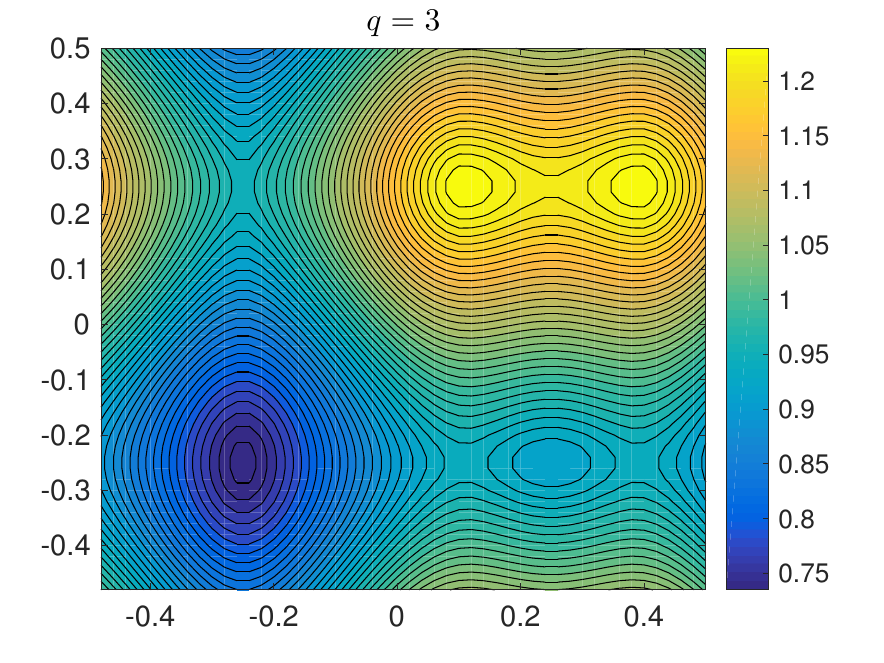}
\includegraphics[width=0.49\textwidth]{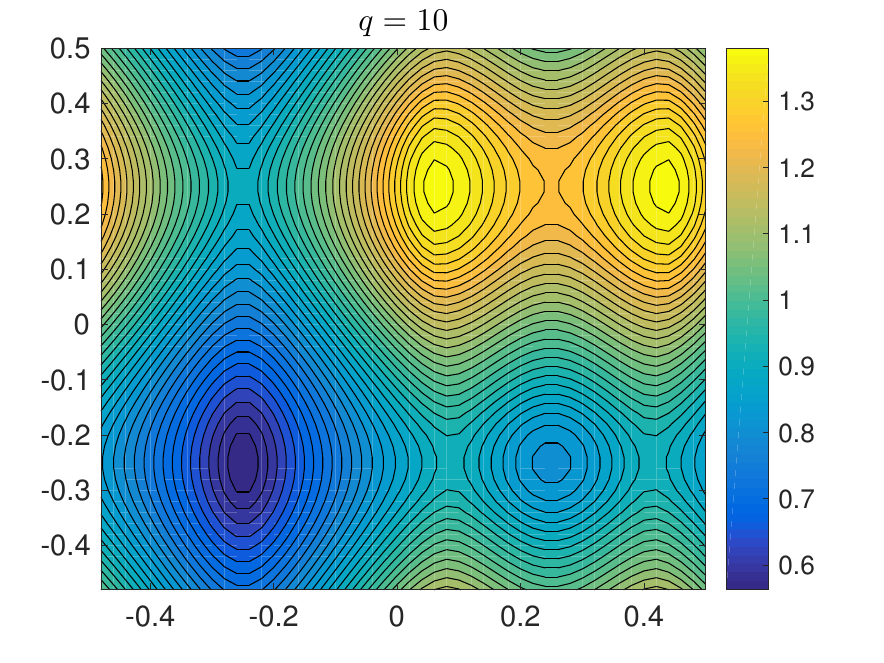}

\caption{Test 4. Contour plots for the unconstrained mass as in the setting of Test 3, with $\nu=1$. CP-U algorithm with $50^2$ nodes. Increasing the value of $q$ generates concentration of mass.}\label{fig:test4}
\end{figure}

\begin{table}[h]
\centering
\setlength{\tabcolsep}{1mm}
  \begin{tabular}{cp{1.7cm}p{1.9cm}p{1.7cm}p{1.7cm}}
  \hline\\
   $q$& Time&   Iterations & $\min m$&$\max m$\\
\hline\\
    1.2   &6.82 [s]  & 11 &   0.9989 &1.0012 \\
    2 &6.70 [s]   & 11 &  0.9072  &1.0737\\
    3&10.57 [s] & 21 &   0.7348 &1.2365\\
    10&24.66 [s] & 57 &  0.5628 &1.3905\\
    \hline 
  \end{tabular}
  \caption{Test 4. Performance for the CP-U algorithm and extremal values of the mass.}\label{tabtest41}
\end{table}

\paragraph{Concluding Remarks.} In this work we have developed proximal methods for the numerical approximation of stationary Mean Field Games systems. The presented schemes perform efficiently in a series of different tests. In particular, the solution through the Chambolle-Pock algorithm is promising in terms of performance, robustness with respect to the viscosity parameter, and accuracy. A natural extension of this work is its application for the approximation of time-dependent case, and the further study of the different features of the approach, which allows constraints on the mass and the modeling of congested transport.

\paragraph{Acknowledgments.} The first author thanks the support of 
CONICYT through grants FONDECYT 11140360, 
MathAmSud 15MATH02, and ECOS-CONICYT C13E03. The second author thanks 
the support of 
ERC-Advanced Grant OCLOC:{\it``From Open-Loop to Closed-Loop Optimal 
Control of PDEs"}. The third author thanks the  support from  project 
iCODE  :{\it``Large-scale systems and Smart grids: distributed 
decision making''} and   from the Gaspar Monge Program for 
Optimization and Operation Research (PGMO). 

\bibliographystyle{plain}
\bibliography{cursocontrolestocastico}

\section*{Appendix}
\subsection{Proof of Theorem~\ref{t:CPockproj}}
\begin{proof} 
In order to simplify the proof we will consider the case 
$\theta=1$ 
as in \cite{CPock11}.
Fix $k\in\NN$, let $\hat{y}\in C$ and $\hat{\sigma}\in\RR^M$ be 
a 
primal-dual solution to \eqref{e:primal}-\eqref{e:dual}.
It follows from $P_C\hat{y}=\hat{y}$ and \eqref{e:caract} that
\begin{align}
\label{e:auxili}
\frac{|y^{k+1}-\hat{y}|^2}{2\tau}+
\frac{|\sigma^{k+1}-\hat{\sigma}|^2}{2\gamma}&\leq
\frac{1}{2\tau}\Big(|p^{k+1}-\hat{y}|^2-|p^{k+1}-y^{k+1}|^2\Big)
+\frac{|\sigma^{k+1}-\hat{\sigma}|^2}{2\gamma}\nonumber\\
&\leq
\frac{1}{2\tau}\Big(|y^{k}-\hat{y}-\tau\Xi^{*}(\sigma^{k+1}-\hat{\sigma})|^2
-|y^{k}-p^{k+1}-\tau\Xi^{*}(\sigma^{k+1}-\hat{\sigma})|^2\nonumber\\
&\hspace{7.5cm}-|p^{k+1}-y^{k+1}|^2\Big)\nonumber\\
&\hspace{.5cm}+\frac{1}{2\gamma}\Big(|\sigma^{k}-\hat{\sigma}+\gamma\Xi(\bar{y}^k-\hat{y})|^2
-|\sigma^{k}-\sigma^{k+1}+\gamma\Xi(\bar{y}^k-\hat{y})|^2\nonumber\\
&=
\frac{1}{2\tau}\Big(|y^{k}-\hat{y}|^2-2\tau\scal{p^{k+1}-\hat{y}}{\Xi^{*}(\sigma^{k+1}-\hat{\sigma})}
-|y^{k}-p^{k+1}|^2\nonumber\\
&\hspace{7.5cm}-|p^{k+1}-y^{k+1}|^2\Big)\nonumber\\
&\hspace{.5cm}+\frac{1}{2\gamma}\Big(|\sigma^{k}-\hat{\sigma}|^2
+2\gamma\scal{\sigma^{k+1}-\hat{\sigma}}{\Xi(\bar{y}^k-\hat{y})}
-|\sigma^{k}-\sigma^{k+1}|^2\nonumber\\
&=
\frac{|y^{k}-\hat{y}|^2}{2\tau}+\frac{|\sigma^{k}-\hat{\sigma}|^2}{2\gamma}
-\frac{|y^{k}-p^{k+1}|^2}{2\tau}-\frac{|\sigma^{k}-\sigma^{k+1}|^2}{2\gamma}\nonumber\\
&\hspace{.5cm}
-\scal{\Xi(p^{k+1}-\hat{y})}{\sigma^{k+1}-\hat{\sigma}}
+\scal{\sigma^{k+1}-\hat{\sigma}}{\Xi(\bar{y}^k-\hat{y})}\nonumber\\
&\hspace{.5cm}
-\frac{|p^{k+1}-y^{k+1}|^2}{2\tau}\nonumber\\
&=
\frac{|y^{k}-\hat{y}|^2}{2\tau}+\frac{|\sigma^{k}-\hat{\sigma}|^2}{2\gamma}
-\frac{|y^{k}-p^{k+1}|^2}{2\tau}-\frac{|\sigma^{k}-\sigma^{k+1}|^2}{2\gamma}\nonumber\\
&\hspace{.5cm}
+\scal{\sigma^{k+1}-\hat{\sigma}}{\Xi(\bar{y}^k-{p}^{k+1})}-\frac{|p^{k+1}-y^{k+1}|^2}{2\tau}.
\end{align}
Since we are assuming, for simplicity, $\theta=1$, we have 
$\bar{y}^k=y^k+p^k-y^{k-1}$, which yields
\begin{align}
\scal{\sigma^{k+1}-\hat{\sigma}}{\Xi(\bar{y}^k-{p}^{k+1})}&=
\scal{\sigma^{k+1}-\hat{\sigma}}{\Xi(y^k-{p}^{k+1})}+
\scal{\sigma^{k+1}-\hat{\sigma}}{\Xi(p^k-{y}^{k-1})}\nonumber\\
&=-\scal{\sigma^{k+1}-\hat{\sigma}}{\Xi({p}^{k+1}-y^k)}+
\scal{\sigma^{k}-\hat{\sigma}}{\Xi(p^k-{y}^{k-1})}\nonumber\\
&\hspace{4.5cm}+\scal{\sigma^{k+1}-\sigma^{k}}{\Xi(p^k-{y}^{k-1})}\nonumber\\
&\leq-\scal{\sigma^{k+1}-\hat{\sigma}}{\Xi({p}^{k+1}-y^k)}+
\scal{\sigma^{k}-\hat{\sigma}}{\Xi(p^k-{y}^{k-1})}\nonumber\\
&\hspace{4.5cm}+\|\Xi\|\,|\sigma^{k+1}-\sigma^{k}|\,|p^k-{y}^{k-1}|\nonumber\\
&\leq-\scal{\sigma^{k+1}-\hat{\sigma}}{\Xi({p}^{k+1}-y^k)}+
\scal{\sigma^{k}-\hat{\sigma}}{\Xi(p^k-{y}^{k-1})}\nonumber\\
&\hspace{2.5cm}+\sqrt{\gamma\tau}\|\Xi\|\frac{|\sigma^{k+1}-\sigma^{k}|^2}{2\gamma}
+\sqrt{\gamma\tau}\|\Xi\|\frac{|p^k-{y}^{k-1}|^2}{2\tau}.
\end{align}
Therefore, from \eqref{e:auxili} we obtain
\begin{align}
\label{casifejer}
\frac{|y^{k+1}-\hat{y}|^2}{2\tau}+
\frac{|\sigma^{k+1}-\hat{\sigma}|^2}{2\gamma}&\leq
\frac{|y^{k}-\hat{y}|^2}{2\tau}+\frac{|\sigma^{k}-\hat{\sigma}|^2}{2\gamma}
-\frac{|y^{k}-p^{k+1}|^2}{2\tau}
+\frac{|{y}^{k-1}-p^k|^2}{2\tau}\nonumber\\
&\hspace{.5cm}
-(1-\sqrt{\gamma\tau}\|\Xi\|)\left(\frac{|{y}^{k-1}-p^k|^2}{2\tau}+
\frac{|\sigma^{k+1}-\sigma^{k}|^2}{2\gamma}\right)
\nonumber\\
&\hspace{.5cm}
-\frac{|p^{k+1}-y^{k+1}|^2}{2\tau}\nonumber\\
&\hspace{.5cm}
-\scal{\sigma^{k+1}-\hat{\sigma}}{\Xi({p}^{k+1}-y^k)}+
\scal{\sigma^{k}-\hat{\sigma}}{\Xi(p^k-{y}^{k-1})}
\end{align}
By calling
\begin{align}
\label{e:ak}
a^k&=\frac{|\sigma^{k+1}-\hat{\sigma}|^2}{2\gamma}+
\scal{\sigma^{k+1}-\hat{\sigma}}{\Xi(y^k-{p}^{k+1})}
+\frac{|y^{k}-p^{k+1}|^2}{2\tau}\nonumber\\
&\geq \frac{|\sigma^{k+1}-\hat{\sigma}|^2}{2\gamma}+
\scal{\sigma^{k+1}-\hat{\sigma}}{\Xi(y^k-{p}^{k+1})}
+\frac{\gamma\|\Xi\|^2|y^{k}-p^{k+1}|^2}{2}\nonumber\\
&\geq \frac{1}{2\gamma}\left(|\sigma^{k+1}-\hat{\sigma}|^2
+2\scal{\sigma^{k+1}-\hat{\sigma}}{\gamma\Xi(y^k-{p}^{k+1})}+
|\gamma\Xi(y^{k}-p^{k+1})|^2\right)\nonumber\\
&=\frac{1}{2\gamma}|\sigma^{k+1}-\hat{\sigma}+\gamma\Xi(y^k-{p}^{k+1})|^2\geq0,
\end{align}
it follows from \eqref{casifejer} that
\begin{align}
\frac{|y^{k+1}-\hat{y}|^2}{2\tau}+a^{k+1}&\leq 
\frac{|y^{k}-\hat{y}|^2}{2\tau}+a^k-(1-\sqrt{\gamma\tau}\|\Xi\|)\left(\frac{|{y}^{k-1}-p^k|^2}{2\tau}+
\frac{|\sigma^{k+1}-\sigma^{k}|^2}{2\gamma}\right)
\nonumber\\
&\hspace{.5cm}
-\frac{|p^{k+1}-y^{k+1}|^2}{2\tau}
\end{align}
and, hence, $(\frac{|y^{k}-\hat{y}|^2}{2\tau}+a^k)_{k\in\NN}$ 
is a  
Fej\'er sequence and, from \cite[Lemma~3.1(iii)]{CombettesQF} we 
have
\begin{equation}
\label{e:tozero}
y^{k-1}-p^k\to 0,\:\sigma^{k+1}-\sigma^k\to 
0,\:p^{k}-y^k\to0
\end{equation}
and  there exists $\alpha\geq0$ such that
$\frac{|y^{k}-\hat{y}|^2}{2\tau}+a^k\to\alpha$. It follows 
from 
\eqref{e:ak} that
\begin{equation}
\left|a^k-\frac{|\sigma^{k+1}-\hat{\sigma}|^2}{2\gamma}\right|\leq
\|\Xi\|\,|\sigma^{k+1}-\hat{\sigma}|\,|y^k-{p}^{k+1}|
+\frac{|y^{k}-p^{k+1}|^2}{2\tau}\to0,
\end{equation}
which yields
\begin{equation}
\label{e:alpha}
\xi^k(\hat{y},\hat{\sigma}):=\frac{|y^{k}-\hat{y}|^2}{2\tau}+\frac{|\sigma^{k+1}-\hat{\sigma}|^2}{2\gamma}
=\frac{|y^{k}-\hat{y}|^2}{2\tau}+a^k+\frac{|\sigma^{k+1}-\hat{\sigma}|^2}{2\gamma}-a^k\to\alpha.
\end{equation}
Hence, we have from \eqref{e:ak} that
$(y^k)_{k\in\NN}$ and $(\sigma^k)_{k\in\NN}$ are bounded.
Let $\bar{y}$ and $\bar{\sigma}$ be accumulation points of the 
sequences $(y^k)_{k\in\NN}$ and $(\sigma^k)_{k\in\NN}$, 
respectively, 
say $y^{k_n}\to \bar{y}$ and $\sigma^{k_n}\to\bar{\sigma}$. It 
follows from \eqref{e:tozero} that 
$\sigma^{k_n+1}\to\bar{\sigma}$,
$p^{k_n}\to \bar{y}$, $p^{k_n+1}\to \bar{y}$, $y^{k_n-1}\to 
\bar{y}$ 
and
$\bar{y}^{k_n}=y^{k_n}+p^{k_n}-y^{k_n-1}\to \bar{y}$. Hence, 
since 
$\prox_{\gamma\psi*}$, $\prox_{\tau\varphi}$, and $P_C$ are 
continuous,
by passing through the limit in \eqref{e:cpockproj}, 
we 
obtain $\bar{y}\in C$ and
\begin{equation}
\begin{cases}
\prox_{\tau\varphi}(\bar{y}-\tau\Xi^{*}\bar{\sigma})=\bar{y}\\
\prox_{\gamma\psi^*}(\bar{\sigma}+\gamma\Xi\bar{y})=\bar{\sigma},
\end{cases}
\end{equation} 
and, from \eqref{e:caract}, $(\bar{y},\bar{\sigma})$ is a 
primal-dual 
solution to \eqref{e:primal}. It is enough to prove that there 
is 
only one accumulation point. By contradiction, suppose that 
$(\bar{y}_1,\bar{\sigma}_1)$ and $(\bar{y}_2,\bar{\sigma}_2)$ 
are two 
accumulation points, say 
$(y^{k_n},\sigma^{k_n})\to(\bar{y}_1,\bar{\sigma}_1)$ and 
$(y^{k_m},\sigma^{k_m})\to(\bar{y}_2,\bar{\sigma}_2)$.
Since any accumulation point is a solution, we deduce 
from \eqref{e:alpha} that there exist $\alpha_1\geq0$ and 
$\alpha_2\geq0$ 
such that $\xi^{k}(\bar{y}_1,\bar{\sigma}_1)\to\alpha_1$ and 
$\xi^{k}(\bar{y}_2,\bar{\sigma}_2)\to\alpha_2$.
Now, for every $k\in\NN$,
\begin{align}
\xi^{k}(\bar{y}_1,\bar{\sigma}_1)&=
\frac{|y^{k}-\bar{y}_1|^2}{2\tau}
+\frac{|\sigma^{k+1}-\bar{\sigma}_1|^2}{2\gamma}\nonumber\\
&=
\xi^{k}(\bar{y}_2,\bar{\sigma}_2)
+\frac{1}{\tau}\scal{y^{k}-\bar{y}_2}{\bar{y}_2-\bar{y}_1}\nonumber\\
&\hspace{1cm}+\frac{1}{\gamma}\scal{\sigma^{k}-\bar{\sigma}_2}{\bar{\sigma}_2
-\bar{\sigma}_1}
+\frac{|\bar{y}_1-\bar{y}_2|^2}{2\tau}
+\frac{|\bar{\sigma}_1-\bar{\sigma}_2|^2}{2\gamma}
\end{align}
Then, we have
\begin{align}
\frac{1}{\tau}\scal{y^{k}}{\bar{y}_2-\bar{y}_1}
+\frac{1}{\gamma}\scal{\sigma^{k}}{\bar{\sigma}_2
-\bar{\sigma}_1}&=\xi^{k}(\bar{y}_1,\bar{\sigma}_1)+
\frac{1}{\tau}\scal{\bar{y}_2}{\bar{y}_2-\bar{y}_1}
+\frac{1}{\gamma}\scal{\bar{\sigma}_2}{\bar{\sigma}_2
-\bar{\sigma}_1}\nonumber\\
&-\frac{|\bar{y}_1-\bar{y}_2|^2}{2\tau}
-\frac{|\bar{\sigma}_1-\bar{\sigma}_2|^2}{2\gamma}-
\xi^{k}(\bar{y}_2,\bar{\sigma}_2)\to\ell,
\end{align}
where 
$\ell:=\alpha_1-\alpha_2+\frac{1}{\tau}\scal{\bar{y}_2}{\bar{y}_2-\bar{y}_1}
+\frac{1}{\gamma}\scal{\bar{\sigma}_2}{\bar{\sigma}_2
-\bar{\sigma}_1}
-\frac{|\bar{y}_1-\bar{y}_2|^2}{2\tau}
-\frac{|\bar{\sigma}_1-\bar{\sigma}_2|^2}{2\gamma}$.
Finally, by taking in particular the subsequences 
$(k_n)_{n\in\NN}$ and 
$(k_m)_{m\in\NN}$ we obtain
\begin{align}
\frac{1}{\tau}\scal{\bar{y}_1}{\bar{y}_2-\bar{y}_1}
+\frac{1}{\gamma}\scal{\bar{\sigma}_1}{\bar{\sigma}_2
-\bar{\sigma}_1}&=
\lim_{n\in\NN}\frac{1}{\tau}\scal{y^{k_n}}{\bar{y}_2-\bar{y}_1}
+\frac{1}{\gamma}\scal{\sigma^{k_n}}{\bar{\sigma}_2
-\bar{\sigma}_1}\nonumber\\
&=\ell\nonumber\\
&=\lim_{m\in\NN}\frac{1}{\tau}\scal{y^{k_m}}{\bar{y}_2-\bar{y}_1}
+\frac{1}{\gamma}\scal{\sigma^{k_m}}{\bar{\sigma}_2
-\bar{\sigma}_1}\nonumber\\
&=\frac{1}{\tau}\scal{\bar{y}_2}{\bar{y}_2-\bar{y}_1}
+\frac{1}{\gamma}\scal{\bar{\sigma}_2}{\bar{\sigma}_2
-\bar{\sigma}_1},
\end{align}
which yields
$|\bar{y}_2-\bar{y}_1|^2/\tau+|\bar{\sigma}_2
-\bar{\sigma}_1|^2/\gamma=0$ and the result follows.
\end{proof}

\end{document}

%% file: macro.tex
\def\dd{{\rm d}}

\def\weight(#1,#2){c_{#1,#2}}

 \if {
  
 } \fi

\if {

}{{\caption}}
} \fi

\def\B{\mathcal{B}}

\def\D{\mathcal{D}}

\def\F{\mathcal{F}}

\def\M{\mathcal{M}}

\def\T{\mathcal{T}}

\def\W{\mathcal{W}}

\def\Y{\mathcal{Y}}
\def\Z{\mathcal{Z}}

\def\Om{{\Omega}}

\def\dom{\mathop{{\rm dom}}}

\def\half{\mbox{$\frac{1}{2}$}}

\def\1B{{\bf  1}}

\newcommand{\NN}{\mathbb{N}}

\newcommand{\RR}{\mathbb{R}}

\newcommand\be{\begin{equation}}
\newcommand\ee{\end{equation}}
\newcommand\ba{\begin{array}}
\newcommand\ea{\end{array}}
\newcommand{\bean}{\begin{eqnarray*}}
\newcommand{\eean}{\end{eqnarray*}}

